\documentclass{article}
\usepackage{hyperref}
\usepackage{graphicx}
\usepackage[all]{xy} 
\usepackage{amsmath}
\usepackage{amsthm}
\usepackage{array}  
\usepackage{amssymb}
\usepackage{calc}
\usepackage{supertabular}
\usepackage{longtable}
\usepackage{txfonts}
\usepackage{textcomp}
\usepackage{colortbl}
\usepackage{tabularx} 
\usepackage[table]{xcolor}
\usepackage{lscape}
\usepackage{rotating}
\usepackage{tensor}
\usepackage{arydshln} 
\usepackage{bbm} 
\usepackage{appendix} 
\usepackage{placeins} 
\usepackage[T1]{fontenc} 
\usepackage{titletoc}
\usepackage{pgf}
\usepackage{tikz}
\usetikzlibrary{matrix,arrows}
\usetikzlibrary{patterns}
\usetikzlibrary{shapes,decorations,shadows}

\DeclareMathOperator{\Cone}{Cone} 
\DeclareMathOperator{\add}{add}  
\DeclareMathOperator{\Mod}{Mod} 
 
\DeclareMathOperator{\Hom}{Hom} 
\DeclareMathOperator{\rad}{rad} 
\DeclareMathOperator{\inj}{inj} 
\DeclareMathOperator{\tor}{tor}

\DeclareMathOperator{\im}{im} 
\DeclareMathOperator{\id}{id}

\DeclareMathOperator{\rep}{rep} 
\DeclareMathOperator{\olp}{olp} 
\DeclareMathOperator{\eolp}{eolp} 
\DeclareMathOperator{\GL}{GL} 
\DeclareMathOperator{\Iso}{Iso} 
\DeclareMathOperator{\mdeg}{mdeg} 
 
\DeclareMathOperator{\codim}{codim} 
\DeclareMathOperator{\Spec}{Spec} 
\newcommand{\Det}{{\det}}
\newcommand*{\punkte}{\dots\unkern}
 
\newcommand{\A}{\mathcal{A}} 
\newcommand{\Fa}{\mathcal{F}} 
\newcommand{\Pa}{\mathcal{P}} 
\newcommand{\Ha}{\mathcal{H}} 
\newcommand{\U}{\mathcal{U}} 
\newcommand{\Q}{\mathcal{Q}} 
\newcommand{\Orb}{\mathcal{O}} 
\newcommand{\N}{\mathcal{N}} 
\newcommand{\V}{\mathcal{V}} 
\newcommand{\quot}{/\!\!/}

\newcommand{\dimv}{\underline{\dim}}
\newcommand{\dfp}{\underline{d}_{P}} 
\newcommand{\df}{\underline{d}} 
\newcommand{\dfs}{\underline{d}_B}

\newtheorem{theorem}{Theorem}[section]
\newtheorem{lemma}[theorem]{Lemma}

\newtheorem{proposition}[theorem]{Proposition}
\newtheorem{corollary}[theorem]{Corollary}

\newtheorem{example}[theorem]{Example}

\begin{document}
\parindent0pt
\title{\bf Conjugation on varieties of nilpotent matrices}

\author{Magdalena Boos\\ Fachbereich C - Mathematik\\ Bergische Universit\"at Wuppertal\\ D - 42097 Wuppertal\\ boos@math.uni-wuppertal.de}
\date{}
\maketitle

\begin{abstract}
We consider the conjugation-action of an arbitrary upper-block parabolic subgroup of $\GL_n(\mathbf{C})$ on the variety of $x$-nilpotent complex matrices.  We obtain a criterion as to whether the action admits a finite number of orbits and specify a system of representatives for the orbits in the finite case of $2$-nilpotent matrices. Furthermore, we give a set-theoretic description of their closures and specify the minimal degenerations in detail for the action of the Borel subgroup. Concerning the action on the nilpotent cone, we obtain a generic normal form of the orbits which yields a $U$-normal form as well, here  $U$ is the standard unipotent subgroup. We describe generating (semi-) invariants for the Borel semi-invariant ring as well as for the $U$-invariant ring. The latter is described in more detail in terms of algebraic quotients by a special toric variety closely related.
\end{abstract}
\section{Introduction}\label{intro}
The study of algebraic group actions on affine varieties, especially the "vertical" study of orbits and their closures, and the "horizontal" study of parametric families of orbits and quotients, are a common topic in algebraic Lie theory.\\[1ex]
A well-known example is the study of the adjoint action of a reductive algebraic group on its Lie algebra and numerous variants thereof, in particular the conjugacy classes of complex (nilpotent) square matrices.\\[1ex]
 Algebraic group actions of reductive groups have particularly been discussed elaborately in connection with orbit spaces and more generally algebraic quotients, even though their application to concrete examples is far from being trivial. In case of a non-reductive group, even most of these results fail to hold true immediately.\\[1ex]
For example, Hilbert's theorem \cite{Hi} yields that for reductive groups, the invariant ring is finitely generated; and a criterion for algebraic quotients is valid \cite{Kr}. In 1958, though, M. Nagata \cite{Na1} constructed a counterexample of a not finitely generated invariant ring corresponding to a non-reductive algebraic group action, which answered Hilbert's fourteenth problem in the negative.\\[1ex]
 One exception are algebraic actions of unipotent subgroups that are induced by reductive groups, since the corresponding invariant ring is always finitely generated \cite{Kr}.\\[1ex]
We turn our main attention towards algebraic non-reductive group actions that are induced by the conjugation action of the general linear group $\GL_n$ over $\mathbf{C}$. For example, the standard parabolic subgroups $P$ (and, therefore, the standard Borel subgroup $B$) and the unipotent subgroup $U$ of $\GL_n$ are not reductive. 
It suggests itself to consider their action on the variety $\N_n^{(x)}$ of $x$-nilpotent matrices of square size $n$ via conjugation which we discuss in this work.\\[1ex]
A recent development in this field is A. Melnikov's study of the $B$-action on the variety of upper-triangular $2$-nilpotent matrices via conjugation \cite{Me1,Me2} motivated by Springer Theory. The detailed description of the orbits and their closures is given in terms of so-called link patterns; these are combinatorial objects visualizing the set of involutions in the symmetric group $S_n$.  In \cite{BoRe}, M. Reineke and the author generalize these results to the Borel-orbits of all $2$-nilpotent matrices and describe the minimal, disjoint degenerations corresponding to their orbit closure relations. Furthermore, L. Fresse describes singularities in the upper-triangular orbit closures by translating the group action to a certain group action on Springer fibres (see \cite{Fr}).\\[1ex]
Another recent outcome is L. Hille's and G. R\"ohrle's study of the action of $P$ on its unipotent radical $P_u$, and on the corresponding Lie algebra $\mathfrak{p}_u$ (see \cite{HiRoe}). They obtain a criterion which varifies that the number of orbits is finite if and only if the nilpotency class of $P_u$ is less or equal than $4$. This result is generalized to all classical groups $G$.\\[1ex]
Given a semi-simple Lie algebra $\mathfrak{g}$ and its Lie group $G$, D. Panyushev considers the adjoint action in \cite{Pan1} and shows that, given a nilpotent element $e\in\mathfrak{g}\backslash\{0\}$, the orbit $G.e$ is spherical if and only if $({\rm ad}_e)^4=0$.  The notion of sphericity translates to $G.e$ admitting only a finite number of Borel-orbits, due to M. Brion \cite{Br1}.\\[1ex]
In this work, we make use of a translation of the classification problem of the $P$-orbits in $\N_n^{(x)}$ to  the description of certain isomorphism classes of representations of a finite-dimensional algebra in Section \ref{transsect}. By making use of this translation, we describe the $P$-orbits in $\N_n^{(2)}$ as well as their closures in detail. Furthermore, we specify all minimal degenerations in Section \ref{x2}. This particular action admits only a finite number of orbits and, by considering $P$-actions on $\N_n^{(x)}$, we find a criterion as to whether the action admits a finite number of orbits in Section \ref{fincrit}.\\[1ex]
 By considering the nilpotent cone $\N$, we generalize the generic $B$-normal form given in \cite{Hal,BoRe} to arbitrary upper-block parabolic subgroups in Section \ref{gnfsect}. We describe $B$-semi-invariants that generate the ring of all $B$-semi-invariants and - as a direct consequence - find $U$-invariants that generate the $U$-invariant ring in Section \ref{generation}. The latter will be made use of to discuss the $U$-invariant ring in more detail in Section \ref{Uquot} by proving a quotient criterion and discussing a toric variety closely related to the algebraic quotient of $\N$ by $U$.\\[1ex]
{\bf Acknowledgments:} The author would like to thank M. Reineke for various valuable discussions concerning the methods and results of this work. Furthermore,  A. Melnikov, K. Bongartz and R. Tange are being thanked for inspirational thoughts and helpful remarks.

\section{Theoretical background}
We denote by $K\coloneqq \mathbf{C}$ the field of complex numbers and by $\GL_n\coloneqq\GL_n(K)$ the general linear group for a fixed integer $n\in\textbf{N}$ regarded as an affine variety. In order to fall back on certain results later on in the article, we include basic knowledge about \mbox{(semi-)} invariants and quotients \cite{Kr,Mu}; and about the representation theory of finite-dimensional algebras \cite{ASS}. 
\subsection{(Semi-) Invariants and Quotients}
 Let $G$ be a linear algebraic group and let $X$ be an affine $G$-variety. We denote by  $X(G)$ the \textit{character group} of $G$; a global section $f\in K[X]$ is called a \textit{$G$-semi-invariant of weight $\chi\in X(G)$} if  $f(g.x)=\chi(g)\cdot f(x)$ for all $x\in X$ and $g\in G$.  \\[1ex]
Let us denote the $\chi$-semi-invariant ring by
\[K[X]^{G}_{\chi}\coloneqq \bigoplus_{n\geq 0}K[X]^{G, n\chi},\]
which is a subring of $K[X]$ and naturally $\textbf{N}$-graded by the sets $K[X]^{G, n\chi}$, that is, by the semi-invariants of weight $n\chi$ (and of degree $n$). The semi-invariant ring corresponding to all characters is denoted by
\[K[X]^{G}_{*}\coloneqq \bigoplus_{\chi\in X(G)}K[X]^{G}_{\chi}.\]
A global section $f\in K[X]$ is called a  \textit{$G$-invariant} if $f(g.x)=f(x)$ for all $x\in X$ and  $g\in G$; the corresponding $G$-invariant ring is denoted by $K[X]^G$.
If the group $G$ is \textit{reductive}, that is, if every linear representation of $G$ can be decomposed into a direct sum of irreducible representations, D. Hilbert showed  that the invariant ring is finitely generated (see \cite{Hi}), even though it can be a problem of large difficulty to find generating invariants.\\[1ex] 
Let $X'$  be yet another affine $G$-variety and let $Y$ be an affine variety.\\[1ex]
A $G$-invariant morphism $\pi\colon X\rightarrow Y=:X\quot G$ is called an \textit{algebraic $G$-quotient of $X$} if 
it fulfills the universal property that for every $G$-invariant morphism $f\colon X\rightarrow Z$, there exists a unique morphism $\hat{f}\colon Y\rightarrow Z$, such that $f=\hat{f}\circ \pi$. If $K[X]^G$ is finitely generated, the variety $\Spec K[X]^G= X\quot G$ induces an algebraic quotient. Each fibre of an algebraic quotient contains exactly one closed orbit, which are, thus, being parametrized.\\[1ex]
In order to calculate an algebraic $G$-quotient of an affine variety, the criterion what follows (see \cite[II.3.4]{Kr}) can be of great help.
\begin{theorem}\label{criterion}
Let $G$ be a reductive group and let $\pi\colon X\rightarrow Y$ be a $G$-invariant morphism of varieties. If 
\begin{enumerate}
 \item $Y$ is normal,
\item $\codim_Y(\overline{Y\backslash \pi(X)})\geq 2$ (or $\pi$ is surjective if $\dim Y=1$) and
\item on a non-empty open subset $Y_0\subseteq Y$ the fibre $\pi^{-1}(y)$ contains exactly one closed orbit for each $y\in Y_0$,
\end{enumerate}
then $\pi$ is an algebraic $G$-quotient of $X$.
\end{theorem}
In case $G$ is not reductive, there are counterexamples of only infinitely generated invariant rings (see \cite{Na1}). One exception are actions of unipotent subgroups induced by reductive group actions, which are discussed in \cite[III.3.2]{Kr}.
\begin{lemma}\label{Uinvfin}
 Let $U$ be a unipotent subgroup of $G$; the action of $G$ restricts to an action of $U$ on $X$. Then the invariant ring $K[X]^U$ is finitely generated as a $K$-algebra.
\end{lemma}

\subsection{Toric varieties}
Since our considerations will involve the notion of a toric variety, we discuss it briefly. For more information on the subject, the reader is referred to \cite{Fu}.\\[1ex]
A \textit{toric variety} is an irreducible variety $X$ which containes $(K^*)^n$ as an open subset, such that the action of $(K^*)^n$ on itself extends to an action of $(K^*)^n$ on $X$.\\[1ex]
Let $N$ be a \textit{lattice}, that is, a free abelian group $N$ of finite rank. By $M:=\Hom_{\textbf{Z}}(N,\textbf{Z})$ we denote the \textit{dual lattice}, together with the induced dual pairing $\langle\_,\_\rangle$.  Consider the vector space $N_{\textbf{R}}:=N\otimes_{\textbf{Z}}\textbf{R}\cong \textbf{R}^n$.\\[1ex]
A subset $\sigma\subseteq N_{\textbf{R}}$ is called a \textit{strongly convex rational polyhedral cone} if $\sigma\cap (-\sigma) =\{0\}$ and if there is a finite set $S\subseteq N$ that generates $\sigma$, that is,
\[\sigma = \Cone(S):= \left\lbrace \sum\limits_{s\in S} \lambda_s\cdot s \mid \lambda_s\geq 0 \right\rbrace. \]
Given a strongly convex rational polyhedral cone $\sigma$, we define its dual by
\[\sigma^{\vee}:=\{m\in \Hom_{\textbf{R}}(\textbf{R}^n,\textbf{R})\mid \langle m,v\rangle\geq 0~\textrm{for~all}~v\in\sigma\}\]
and its corresponding additive semigroup by $S_{\sigma}:=\sigma^{\vee}\cap M$, which is finitely generated due to Gordon's lemma (see \cite{Fu}). Note that  if $\sigma$ is a maximal dimensional strongly convex rational polyhedral cone, then $\sigma^{\vee}$ is one as well. We associate to it the semigroup algebra $KS_{\sigma}$ and obtain an affine toric variety $\Spec KS_{\sigma}$. 
\begin{lemma}\label{toricco}
 An affine toric variety $X$ is isomorphic to $\Spec KS_{\sigma}$ for some strongly convex rational polyhedral cone $\sigma$ if and only if $X$ is normal.
\end{lemma}

\subsection{Representation theory of finite-dimensional algebras}\label{repsofalgebras}
A \textit{finite quiver} $\Q$ is a directed graph $\Q=(\Q_0,\Q_1,s,t)$, such that $\Q_0$ is a finite set of \textit{vertices} and $\Q_1$ is a finite set of \textit{arrows}, whose elements are written as $\alpha\colon s(\alpha)\rightarrow t(\alpha)$.
The \textit{path algebra} $K\Q$ is defined as the $K$-vector space with a basis consisting of all paths in $\Q$, that is, sequences of arrows $\omega=\alpha_s\punkte\alpha_1$, such that $t(\alpha_{k})=s(\alpha_{k+1})$ for all $k\in\{1,\punkte,s-1\}$; formally included is a path $\varepsilon_i$ of length zero for each $i\in \Q_0$ starting and ending in $i$. The multiplication is defined by
\begin{center}
 $\omega\cdot\omega'=\left\{\begin{array}{ll}\omega\omega',&~\textrm{if}~t(\beta_t)=s(\alpha_1);\\
0,&~\textrm{otherwise.}\end{array}\right.$\end{center}
where $\omega\omega'$ is the  concatenation of paths $\omega$ and $\omega'$.\\[1ex]
We define the \textit{radical} $\rad(K\Q)$ of $K\Q$ to be the (two-sided) ideal generated by all paths of positive length; then an arbitrary ideal $I$ of $K\Q$ is called \textit{admissible} if there exists an integer $s$ with $\rad(K\Q)^s\subset I\subset\rad(K\Q)^2$.\\[1ex]
A finite-dimensional \textit{$K$-representation} of $\Q$ is a tuple \[((M_i)_{i\in \Q_0},(M_\alpha\colon M_i\rightarrow M_j)_{(\alpha\colon i\rightarrow j)\in \Q_1}),\] where the $M_i$ are $K$-vector spaces, and the $M_{\alpha}$ are $K$-linear maps.\\[1ex]
 A \textit{morphism of representations} $M=((M_i)_{i\in \Q_0},(M_\alpha)_{\alpha\in \Q_1})$ and
 \mbox{$M'=((M'_i)_{i\in \Q_0},(M'_\alpha)_{\alpha\in \Q_1})$} consists of a tuple of $K$-linear maps $(f_i\colon M_i\rightarrow M'_i)_{i\in \Q_0}$, such that $f_jM_\alpha=M'_\alpha f_i$ for every arrow $\alpha\colon i\rightarrow j$ in $\Q_1$.\\[1ex]
For a representation $M$ and a path $\omega$ in $\Q$ as above, we denote $M_\omega=M_{\alpha_s}\cdot\punkte\cdot M_{\alpha_1}$. A representation $M$ is called \textit{bound by $I$} if $\sum_\omega\lambda_\omega M_\omega=0$ whenever $\sum_\omega\lambda_\omega\omega\in I$.\\[1ex]
These definitions yield certain categories as follows: We denote by $\rep_K(\Q)$ the abelian $K$-linear category of all representations of $\Q$ and by   $\rep_K(\Q,I)$ the category of representations of $\Q$ bound by $I$; the latter is equivalent to the category of finite-dimensional $K\Q/I$-representations.\\[1ex]
Given a representation $M$ of $\Q$, its \textit{dimension vector} $\dimv M\in\mathbf{N}\Q_0$ is defined by $(\dimv M)_{i}=\dim_k M_i$ for $i\in \Q_0$. Let us fix a dimension vector $\df\in\mathbf{N}\Q_0$, then we denote by $\rep_K(\Q,I)(\df)$ the full subcategory of $\rep_K(\Q,I)$ which consists of representations of dimension vector $\df$.\\[1ex]
For certain classes of finite-dimensional algebras, a convenient tool for the classification of the indecomposable representations is the \textit{Auslander-Reiten quiver} $\Gamma(\Q,I)$ of $\rep_K(\Q,I)$. Its vertices $[M]$ are given by the isomorphism classes of indecomposable representations of $\rep_K(\Q,I)$; the arrows between two such vertices $[M]$ and $[M']$ are parametrized by a basis of the space of so-called irreducible maps $f\colon M\rightarrow M'$.\\[1ex] 
One standard technique to calculate the Auslander-Reiten quiver is the \textit{knitting process} (see, for example, \cite[IV.4]{ASS}).
 In some cases, the Auslander-Reiten quiver $\Gamma(\Q,I)$ can be calculated by using \textit{covering techniques} (see \cite{Ga3} or \cite{BoGa}).\\[1ex]
By defining the affine space $R_{\df}(\Q):= \bigoplus_{\alpha\colon i\rightarrow j}\Hom_K(K^{d_i},K^{d_j})$, one realizes that its points $m$ naturally correspond to representations $M\in\rep_K(\Q)(\df)$ with $M_i=K^{d_i}$ for $i\in \Q_0$. 
 Via this correspondence, the set of such representations bound by $I$ corresponds to a closed subvariety $R_{\df}(\Q,I)\subset R_{\df}(\Q)$.\\[1ex]
The algebraic group $\GL_{\df}=\prod_{i\in \Q_0}\GL_{d_i}$ acts on $R_{\df}(\Q)$ and on $R_{\df}(\Q,I)$ via base change, furthermore the $\GL_{\df}$-orbits $\Orb_M$ of this action are in bijection to the isomorphism classes of representations $M$ in $\rep_K(\Q,I)(\df)$.
There is an induced $\GL_{\df}$-action on $K[R_{\df}(\Q)]$ which yields the natural notion of a semi-invariant. \\[1ex]
Let us denote by $\add\Q$ the \textit{additive category} of $\Q$ with objects $O(i)$ corresponding to the vertices $i\in\Q_0$ and morphisms induced by the paths in $\Q$. Since every representation $M\in \rep_K(\Q)$ can naturally be seen as  a functor from $\add \Q$ to $\Mod K$, we denote this functor by $M$ as well.
 Let $\phi\colon\bigoplus_{i=1}^n O(i)^{x_i}\rightarrow \bigoplus_{i=1}^n O(i)^{y_i}$ be an arbitrary morphism in $\add \Q$ and consider $\df\in \mathbf{N}\Q_0$, such that $\sum_{i\in \Q_0}x_i\cdot \df_i=\sum_{i\in \Q_0}y_i\cdot \df_i$. An induced so-called determinantal semi-invariant is given by
\[ f_{\phi}\colon~R_{\df}(\Q)\rightarrow K ;~~ m~~\mapsto \det(M(\phi)), \]
where $m\in R_{\df}(\Q)$ and $M\in  \rep_K(\Q)(\df)$ are related via the above mentioned correspondence.
 The following theorem (see \cite{SvB}) is due to A. Schofield and M. van den Bergh.
\begin{theorem}\label{Schosemi}
 The semi-invariants in $K[R_{\df}(\Q)]^{\GL_{\df}}_*$ are spanned by the determinantal semi-invariants $f_{\phi}$.
\end{theorem}
 \section{Translation to a representation-theoretic setup}\label{transsect}
We fix a parabolic subgroup $P$ of $\GL_n$ of block sizes $(b_1,\punkte,b_p)$.\\[1ex] 
We define  $\Q_p$ to be the quiver
\begin{center}\begin{tikzpicture}
\matrix (m) [matrix of math nodes, row sep=0.01em,
column sep=1.5em, text height=0.5ex, text depth=0.1ex]
{\Q_p\colon & \bullet & \bullet &  \bullet & \cdots  & \bullet & \bullet  & \bullet \\ & \mathrm{1} & \mathrm{2} &  \mathrm{3} & &   \mathrm{p-2} &  \mathrm{p-1}  & \mathrm{p} \\ };
\path[->]
(m-1-2) edge node[above=0.05cm] {$\alpha_1$} (m-1-3)
(m-1-3) edge  node[above=0.05cm] {$\alpha_2$}(m-1-4)
(m-1-6) edge  node[above=0.05cm] {$\alpha_{p-2}$}(m-1-7)
(m-1-7) edge node[above=0.05cm] {$\alpha_{p-1}$} (m-1-8)
(m-1-8) edge [loop right] node{$\alpha$} (m-1-8);\end{tikzpicture}\end{center} 
and $\A(p,x)\coloneqq K \Q_p/I_x$  to be the finite-dimensional algebra, where $I_x\coloneqq (\alpha^x)$ is an admissible ideal. We fix the dimension vector 
\[\dfp\coloneqq(d_1,\punkte,d_p)\coloneqq(b_1,b_1+b_2, \punkte, b_1+...+b_p)\]
 and formally set $b_0=0$. As explained in Section \ref{repsofalgebras}, the algebraic group $\GL_{\dfp}$ acts on $R_{\dfp}(\Q_p,I_x)$; the orbits of this action are in bijection with the isomorphism classes of representations in $\rep_{K}(\Q_p,I_x)(\dfp)$.\\[1ex]
Let us define $\rep_{K}^{\inj}(\Q_p,I_x)(\dfp)$ to be the full subcategory of $\rep_{K}(\Q_p,I_x)(\dfp)$ consisting of representations $((M_i)_{1\leq i\leq p},(M_{\rho})_{\rho\in \Q_1})$, such that $M_{\rho}$ is injective if $\rho=\alpha_i$ for every $i\in\{1,\punkte, p-1\}$. Corresponding to this subcategory, there is an open subset $R_{\dfp}^{\inj}(\Q_p,I_x)\subset R_{\dfp}(\Q_p,I_x)$, which is stable under the $\GL_{\dfp}$-action.\\[1ex]
We denote $\Orb_M:=\GL_{\dfp}.m$ if $m\in R_{\dfp}^{\inj}(\Q_p,I_x)$ corresponds to $M\in\rep^{\inj}(\Q_p,I_x)(\dfp)$ as in Section \ref{repsofalgebras}.
In order to describe the orbit closure $\overline{\Orb_M}$, we denote $M\leq_{\deg}M'$ if $\Orb_{M'}\subset\overline{\Orb_M}$ in $R_{\df}(\Q,I)$ for a representation $M'$ and say that $M'$ is a \textit{degeneration} of $M$. 
 Of course, in order to describe all degenerations, it is sufficient to calculate all \textit{minimal degenerations} $M<_{\mdeg}M'$, that is, degenerations $M<_{\deg}M'$, such that if $M\leq_{\deg}L\leq_{\deg}M'$, then $M\cong L$ or $M'\cong L$.\\[1ex]
The following lemma is a slightly generalized version of \cite[Lemma 3.2]{BoRe}. The proof is similar, though.
\begin{lemma} \label{bijection}
There is an isomorphism $R_{\dfp}^{\inj}(\Q_p,I_x)\cong \GL_{\dfp}\times^{P}\N_n^{(x)}$. Thus, there exists a bijection $\Phi$ between the set of $P$-orbits in $\N_n^{(x)}$ and the set of $\GL_{\dfp}$-orbits in $R_{\dfp}^{\inj}(\Q_p,I_x)$, which sends an orbit $P.N\subseteq \N^{(x)}$ to the isomorphism class of the representation
\begin{center}\begin{tikzpicture}
\matrix (m) [matrix of math nodes, row sep=0.05em,
column sep=2em, text height=1.5ex, text depth=0.2ex]
{ K^{d_1} & K^{d_2} & K^{d_3} & \cdots  & K^{d_{p-2}} & K^{d_{p-1}}  & K^{n}\\ };
\path[->]
(m-1-1) edge node[above=0.05cm] {$\epsilon_1$} (m-1-2)
(m-1-2) edge  node[above=0.05cm] {$\epsilon_2$}(m-1-3)
(m-1-3) edge  (m-1-4)
(m-1-4) edge  (m-1-5)
(m-1-5) edge  node[above=0.05cm] {$\epsilon_{p-2}$}(m-1-6)
(m-1-6) edge node[above=0.05cm] {$\epsilon_{p-1}$} (m-1-7)
(m-1-7) edge [loop right] node{$N$} (m-1-7);\end{tikzpicture}\end{center}
 (denoted $M^N$) with natural embeddings $\epsilon_i\colon K^i\hookrightarrow K^{i+1}$. This bijection preserves orbit closure relations,
 dimensions of stabilizers (of single points) and codimensions.
\end{lemma}
Due to considerations of different parabolic subgroups and nilpotency degrees, the classification of the corresponding isomorphism classes of representations differs wildly.
\section[P-conjugation on the variety of 2-nilpotent matrices]{$P$-conjugation on $\N_n^{(2)}$}\label{x2}
Let us consider the action of $P$ on the variety $\N_n^{(2)}$ of $2$-nilpotent $n\times n$- matrices. \\[1ex]
As the theorem of W. Krull, R. Remak and O. Schmidt states, every representation in $\rep_{K}(\Q_p,I_2)$ can be decomposed into a direct sum of indecomposables, which is unique up to permutations and isomorphisms. Following \cite[Theorem 3.3]{BoRe}, the following lemma classifies the indecomposables in $\rep_{K}(\Q_p,I_2)$.
\begin{lemma}\label{indec2nilp}
 Up to isomorphisms, the indecomposable representations in $\rep_{K}^{\inj}(\Q_p,I_x)$ are the following:\\[1ex]
$\U_{i,j}$ for $1\leq j\leq i\leq p$:
\[\begin{tikzpicture}
\matrix (m) [matrix of math nodes, row sep=0.02em,
column sep=0.08em, text height=1.0ex, text depth=0.25ex]
{ 0 & \xrightarrow{0} & \cdots & \xrightarrow{0} & 0 & \xrightarrow{0} & K & \xrightarrow{id} & \cdots & \xrightarrow{id} & K & \xrightarrow{e_1} & K^{2} & \xrightarrow{id}& \cdots & \xrightarrow{id} & K^{2}  \\};
\path[->]
(m-1-17) edge [loop right] node{$\alpha$} (m-1-17);
\end{tikzpicture}\]
$\U_{i,j}$ for $1\leq i< j\leq p$:
\[\begin{tikzpicture}
\matrix (m) [matrix of math nodes, row sep=0.02em,
column sep=0.08em, text height=1.0ex, text depth=0.25ex]
{  0 & \xrightarrow{0} & \cdots & \xrightarrow{0} & 0 & \xrightarrow{0} & K & \xrightarrow{id} & \cdots & \xrightarrow{id} & K & \xrightarrow{e_2} & K^{2} & \xrightarrow{id}& \cdots & \xrightarrow{id} & K^{2}  \\ };
\path[->]
(m-1-17) edge [loop right] node{$\alpha$} (m-1-17);
\end{tikzpicture}\]
$\V_{i}$ for $1\leq i\leq p$:
\[\begin{tikzpicture}
\matrix (m) [matrix of math nodes, row sep=0.02em,
column sep=0.1em, text height=1.0ex, text depth=0.25ex]
{  0 & \xrightarrow{0} & \cdots & \xrightarrow{0} & 0 & \xrightarrow{0} & K & \xrightarrow{id} & \cdots & \xrightarrow{id} & K  \\};
\path[->]
(m-1-11) edge [loop right] node{$0$} (m-1-11);
\end{tikzpicture}\]
Here, $e_1$ and $e_2$ are the standard coordinate vectors of $K^2$ and $\alpha\cdot e_1=e_2$, $\alpha\cdot e_2=0$.
\end{lemma}
 An \textit{enhanced oriented link pattern} of type $(b_1,\punkte,b_p)$ is an oriented graph on the vertices $\{1,\punkte ,p\}$ together with a (possibly empty) set of of dots at each vertex, such that the sum of the numbers of sources, targets and dots at every vertex $i$ equals $b_i$. Clearly, an enhanced oriented link pattern of a fixed type is far from being unique.\\[1ex]
For example, an enhanced oriented link pattern of type $(3,2,6,2,5)$ is given by
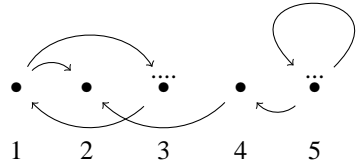
\begin{center}\begin{tikzpicture}[descr/.style={fill=white,inner sep=2.5pt}]
  \matrix (m) [matrix of math nodes, row sep=1.01em, column sep=1.5em, text height=1.5ex, text depth=0.25ex]
 {\bullet & \bullet & \ddddot{\bullet} & \bullet & \dddot{\bullet} \\
1 & 2 & 3 & 4 &5\\};
\path[->,font=\scriptsize]
(m-1-1) edge [bend left=60]  (m-1-3)
(m-1-3) edge [bend left=45]   (m-1-1)
(m-1-4) edge [bend left=45]   (m-1-2)
(m-1-1) edge [bend left=45]   (m-1-2)
(m-1-5) edge [loop]   (m-1-5)
(m-1-5) edge [bend left=45]   (m-1-4);
  \end{tikzpicture}.\end{center} 

\begin{theorem}\label{paraboliccase}
There are natural bijections between 
\begin{enumerate}
\item $P$-orbits in $\N_n^{(2)}$,
\item  isomorphism classes in $\rep^{\inj}_{K}(\Q_p,I_2)$ of dimension vector $\dfp$,
\item matrices $N=(p_{i,j})_{i,j}\in \mathbf{N}^{p\times p}$, such that
$\sum_j (p_{i,j}+ p_{j,i}) \leq b_i$ for all $i\in\{1,\punkte,p\}$,
\item   and enhanced oriented link patterns of type $(b_1,\punkte,b_p)$.
\end{enumerate}
Moreover, if the isomorphism class of $M$ corresponds to a matrix $N$ under this bijection, the orbit $\Orb_M\subset R_{\dfp}^{\inj}(\Q_p,I_2)$ and the orbit $P.N\subset\N_n^{(2)}$ correspond to each other via the bijection $\Phi$ of Lemma \ref{bijection}.
\end{theorem}
The proof is similar to the proof of \cite[Theorem 3.4]{BoRe}.
Note that the multiplicity of the indecomposable $\V_i$ is obtained as the number of dots at the vertex $i$ which we call ``fixed vertices''. The multiplicity of the indecomposable $\U_{i,j}$ is given as the number of arrows $j\rightarrow i$. 
We define $\eolp(X)$ to be the enhanced oriented link pattern corresponding to both the isomorphism class of \mbox{$X\in\rep^{\inj}_{K}(\Q_p,I_2)(\dfp)$} and the $P$-orbit of $X\in\N_n^{(2)}$.\\[1ex] 
An \textit{oriented link pattern} of size $n$ is an  enhanced oriented link pattern of type $(1,\punkte ,1)$. Thus, every vertex is incident with at most one arrow.
The concrete classification of the Borel-orbits is then given by the  oriented link patterns of size $n$ and is easily obtained from Theorem \ref{paraboliccase} (see, for the detailed proof, \cite[Theorem 3.4]{BoRe}). 
As before, we define $\olp(X)$ to be the oriented link pattern corresponding to both the isomorphism class of \mbox{$X\in\rep^{inj}_{K}(\Q_n,I_2)(\dfs)$} and the $B$-orbit of $X\in\N_n^{(2)}$.\\[1ex]
Given representations $M,M'\in \rep_{K}(\Q_p,I_2)$, we set $[M, M'] \coloneqq \dim_{K} \Hom(M, M' )$.
These dimensions are calculated for indecomposable representations in $\rep_{K}(\Q_p,I_2)$ in \cite[Lemma 4.2]{BoRe}.
\begin{proposition}\label{dimhom}
 Let $i, j, k, l \in \{1,\punkte, p\}$. Then
\begin{enumerate}
\item $[\V_k , \V_i ] = [\V_k , \U_{i,j} ] =\delta_{i\leq k}$,
\item $[\U_{k,l} , \V_i ] = \delta_{i\leq l}$,
\item $[\U_{k,l} , \U_{i,j} ] = \delta_{i\leq l} + \delta_{j\leq l} \cdot \delta_{i\leq k}$,
\end{enumerate}
where
$\delta_{x\leq y} \coloneqq \left\{ \begin{array}{ll} 1, & \hbox{$ \textrm{if}~ x\leq y$;} \\ 0, & \hbox{$ \textrm{otherwise}$.} \end{array} \right. $ 
\end{proposition}
In order to prove an easy description of the parabolic orbit closures in $\N_n^{(2)}$ in terms of (enhanced) oriented link patterns, we discuss how the dimensions of Proposition \ref{dimhom} are linked with these.
\begin{proposition}\label{combihom}
 Let $M\in \rep^{\inj}_{K}(\Q_n,I_2)(\dfp)$ and let $i, j, k, l \in \{1,\punkte, p\}$. Then by considering $X:=\eolp(M)$:\\[1ex]
 1. $a_k(M)\coloneqq[\V_k,M]= \sharp\{ fixed~vertices~\leq k~{\rm in}~X\}~+~\sharp\{targets~of~arrows~\leq k~{\rm in}~X\}$, \\[1ex]
2. $b_{k,l}(M)\coloneqq[\U_{k,l},M]= a_l(M) + \sharp\{ \textrm{arrows~with~source}~\leq l~  \textrm{and~target}~\leq k~{\rm in}~X\}$,\\[1ex]
3. $\overline{a_i}(M)\coloneqq[M,\V_i]= \sharp\{ \textrm{fixed~vertices}~\geq i~{\rm in}~X\}~+~\sharp\{ \textrm{sources~of~arrows}~\geq i~{\rm in}~X\}$,\\[1ex]
4. $\overline{b_{i,j}}(M)\coloneqq[M,\U_{i,j}]= \overline{a_i}(M) + \sharp\{ \textrm{arrows~with~source}~\geq j~  \textrm{and~target}~\geq i~{\rm in}~X\}$.
\end{proposition}
For two representations $M=\bigoplus_{i,j=1}^p\U_{i,j}^{m_{i,j}}\oplus \bigoplus_{i=1}^p\V_i^{n_i}$ and $M'=\bigoplus_{i,j=1}^p\U_{i,j}^{m'_{i,j}}\oplus \bigoplus_{i=1}^p\V_i^{n'_i}$ in $\rep^{\inj}_{K}(\Q_n,I_2)(\dfp)$, we obtain
\[[M,M']=\sum_{i,j=1}^p m_{i,j}b_{i,j}(M')+ \sum_{k=1}^p n_k a_k(M')=\sum_{i,j=1}^p m'_{i,j}\overline{b_{i,j}}(M)+ \sum_{k=1}^p n'_k \overline{a_k}(M).\]
Let $N\in\N_n^{(2)}$ be a $2$-nilpotent matrix that corresponds to the representation $M$ via the bijection of Lemma \ref{bijection}.
\begin{proposition}\label{stabpar}
\[\dim \overline{P.N}=\dim P.N=\sum\limits_{i=1}^p \sum\limits_{x=1}^i (b_i\cdot b_x)-\sum\limits_{i,j=1}^p m_{i,j}b_{i,j}(N)-\sum\limits_{i=1}^p n_i a_i(N). \]
\end{proposition}
\begin{proof}The equalities
\[ \dim P.N~= \dim P-\dim\Iso_{P}(N) = \dim P- \dim\Iso_{\GL_{\dfp}}(M) = \dim P - [M,M] \]
yield the claim, here $\Iso_{\GL_{\dfp}}(M)$ is the isotropy group of $m\in R_{\dfp}^{\inj}(\Q_p,I_2)$ in $\GL_{\dfp}$.
\end{proof}
\subsection{Borel-orbit closures}
Let $M$ and $M'$ be two representations in $\rep_{K}(\Q_n,I_2)$ of the same dimension vector $\df$.  Since the correspondence of Lemma \ref{bijection} preserves orbit closure relations, we know that $M\leq_{\rm deg}M'$ if and only if the corresponding $2$-nilpotent matrices, denoted by $N=(m_{i,j})_{i,j}$ and $N'=(m'_{i,j})_{i,j}$, respectively, fulfill $B.N'\subset\overline{B.N}$ in $\N_n^{(2)}$. The following theorem can be found in \cite[Theorem 4.3]{BoRe}.
\begin{theorem}\label{pq} We have $M\leq_{\rm deg} M'$ (or equivalently, $B.N'\subset\overline{B.N}$ in the notation above) if and only if $a_k(M)\leq a_k(M')$ and $b_{k,l}(M)\leq b_{k,l}(M')$ for all $k,l\in\{1,\punkte,n\}$.
\end{theorem}
 The key to calculating all minimal degenerations is obtained by the following proposition (see \cite[Corollary 4.5]{BoRe}).
\begin{proposition}\label{mindisj}
Let $D<_{\mdeg}D'$ be a minimal, disjoint degeneration in $\rep^{\inj}_{K}(\Q_n,I_2)$. Then either $D'$ is indecomposable or $D'\cong U\oplus V$, where $U$ and $V$ are indecomposables and there exists an exact sequence $0\rightarrow U\rightarrow D\rightarrow V\rightarrow 0$ or $0\rightarrow V\rightarrow D\rightarrow U\rightarrow 0$.
\end{proposition}
A method to construct all orbits contained in a given orbit closure is described in \cite[Theorem 4.6]{BoRe}, since Proposition \ref{mindisj} ``localizes`` the problem to sequences of changes at at most four vertices of the corresponding oriented link pattern. All these minimal, disjoint degenerations are explicitly listed (in terms of oriented link patterns as well) in \cite[Theorem 4.6]{BoRe}.\\[1ex]
Consider an arbitrary minimal, disjoint degeneration $D<_{\mdeg}D'$ in $\rep^{\inj}_{K}(\Q_n,I_2)$. To classify the minimal degenerations in $\rep^{\inj}_{K}(\Q_n,I_2)(\dfs)$, let us (if possible) consider a representation $W$, such that $D\oplus W<_{\deg}D'\oplus W$ is a degeneration in $\rep^{\inj}_{K}(\Q_n,I_2)(\dfs)$. We give an explicit criterion as to whether this degeneration is minimal. 
\begin{theorem}\label{indecgeneral}
The degeneration $D\oplus W<_{\deg}D'\oplus W$ is minimal if and only if every indecomposable direct summand $X$ of $W$ fulfills $[X,D]-[X,D']=0$ and $[D,X]-[D',X]=0$.
\end{theorem}
\begin{proof}
If the degeneration is obtained from an extension as in Proposition \ref{mindisj}, the claim follows from \cite[Theorem 4]{Bo1} (see the exact argumentation in \cite[Theorem 3.3.11]{B1}).\\[1ex]
Assume that the minimal, disjoint degeneration $D<_{\mdeg}D'$ is given by $\U_{s,t}<_{\mdeg}\U_{t,s}$, that is, by the only minimal, disjoint degeneration not obtained from an extension.\\[1ex]  
The theorem then reads as follows: $D\oplus W<_{\deg}D'\oplus W$ is minimal if and only if every indecomposable direct summand $\V_k$ of $W$ fulfills $\delta_{s<k<t}=0$ and if every indecomposable direct summand $\U_{k,l}$ of $W$ fulfills $\delta_{k<t}\delta_{s<l<t}+\delta_{s<k<t}\delta_{t<l}=0$.\\[1ex]
If $s<k<t$, then the degeneration $\U_{t,s}\oplus \V_{k}<_{\mdeg}\U_{s,t}\oplus\V_{k}$ is not minimal since
\[\U_{t,s}\oplus\V_{k} <_{\deg} \U_{k,s}\oplus\V_t <_{\deg}\U_{s,t}\oplus\V_{k}\]
are proper degenerations.\\[1ex]
If $s\neq k< t$ and $s< l<t$  (or $s<k < t$ and $l>t$, respectively), then the degeneration $\U_{t,s}\oplus \U_{k,l}<_{\deg}\U_{s,t}\oplus\U_{k,l}$ is not minimal, since 
\[\U_{t,s}\oplus\U_{k,l} <_{\deg} \U_{k,s}\oplus\U_{t,l} <_{\deg}\U_{s,t}\oplus\U_{k,l}\]
\[(\U_{t,s}\oplus\U_{k,l} <_{\deg} \U_{k,s}\oplus\U_{t,l} <_{\deg}\U_{s,t}\oplus\U_{k,l},~ \textrm{respectively})\] 
are proper degenerations.\\[1ex]
Consider $W\in\rep^{\inj}_{K}(\Q_n,I_2)$, such that $M\coloneqq\U_{t,s}\oplus W<_{\deg}M'\coloneqq\U_{s,t}\oplus W$ in $\rep^{\inj}_{K}(\Q_n,I_2)(\dfs)$ and such that every direct summand of $W$ fulfills the assumptions.\\[1ex] If the degeneration $M<_{\deg}M'$ is not minimal, then there exists a representation $L$ fulfilling $M<_{\deg}L<_{\deg}M$. Without loss of generality, we can assume $M<_{\mdeg}L$.\\[1ex]
Then $[\V_k,M]\leq [\V_k,L]\leq [\V_k,M']$ for all $k$ and we can translate the statement as follows: The source vertices to the left of $s-1$ and to the right of $t$ coincide in $\olp(M)$, $\olp(L)$ and $\olp(M')$. Also, the number of arrows coincides in all three link patterns, since  $[\V_n,M]= [\V_n,L]= [\V_n,M']$.\\[1ex]
\textbf{Claim 1}:
 Let $\U_{k,l}$ be a direct summand of $M$, $L$ or $M'$. If $l<s$ or ($k<s $ and $l> t$) or ($k>t$ and $l>t$), then $\U_{k,l}$ is a direct summand of $M$, $L$ and $M'$.\\[1ex]
The proof of Claim 1 follows directly from Corollary \ref{combihom}.\\[1ex]
\textbf{Claim 2}:
Let $\U_{k,l}$ be a direct summand of $M$, $L$ or $M'$. If $t< k$ and $s< l<t$, then $\U_{k,l}$ is a direct summand of $M$, $L$ and $M'$.
\begin{proof}[Proof of Claim 2]
Let $t< k$ and $s< l<t$ for two integers $k$ and $l$.\\[1ex]
First, we assume that $U_{k,l}$ is a direct summand of $M$, but not a direct summand of $L$. 
 Since $M<_{\mdeg}L$, the indecomposable $\U_{k,l}$ must be changed by some minimal, disjoint part of the degeneration. The only possibilities for a change like that are the following:\\[1ex]
\textbf{1st case:} The indecomposable $\U_{k',l}$ is a direct summand of $L$, such that $k\neq k'$.\\[0.4ex]
1.1. The minimal, disjoint part is $\U_{k,l}\oplus\V_{k'}<_{\mdeg}\U_{k',l}\oplus \V_k$, such that $k'<k$:\\[0.2ex]
 The indecomposable $\V_{k'}$ can only be a direct summand of $M$ if $k'<s$ or $k'>t$.\\
 If $k'<s$, we obtain $[\U_{k',t},M]<[\U_{k',t},L]$ and if $k'>t$, we obtain $[\U_{k',l},M]<[\U_{k',l},L]$, a contradiction.\\[0.4ex]
1.2. The minimal, disjoint part is $\U_{k,l}\oplus\U_{k',l'}<_{\mdeg}\U_{k',l}\oplus \U_{k,l'}$, such that $k<k'$ and $l'<l$, or such that $k'<k$ and $l<l'$:\\[0.2ex]
 The indecomposable $\U_{k',l'}$ can only be a direct summand of $M$ if $k'>t$ or $l'<s$, or if $k'<s$ and $l'>t$. As has been shown in claim 1, every indecomposable $\U_{i,j}$ with $j<s$, or with $j>t$ and $i<s$ is either a direct summand of $M$, $L$ and $M'$ or a direct summand of none of them. Thus, $k'>t$ and if $k<k'$ and $l'<l$, we obtain $[\U_{k,l'},M]<[\U_{k,l'},L]$. If $k'<k$ and $l<l'$, we obtain $[\U_{k',l},M]<[\U_{k',l},L]$, a contradiction.\\[0.4ex]
1.3. The minimal, disjoint part is $\U_{k,l}\oplus\U_{l',k'}<_{\mdeg}\U_{l',k}\oplus \U_{k',l}$:\\[0.2ex]
 The indecomposable $\U_{l',k'}$ can only be a direct summand of $M$ if $l'>t$ or $k'<s$, or if $l'<s$ and $k'>t$. As has been shown in claim 1, every indecomposable $\U_{i,j}$ with $j<s$, or with $j>t$ and $i<s$ is either a direct summand of $M$, $L$ and $M'$ or a direct summand of none of them.\\[1ex] Thus, $l'>t$ and the only cases possible are $l<l'<k'<k$ and $l<k'<k<l'$. We immediately obtain $[\U_{k',l},M]<[\U_{k',l},L]$, a contradiction.\\[0.4ex]
\textbf{2nd case:} The indecomposable $\U_{k,l'}$ is a direct summand of $L$, such that $l\neq l'$.\\[0.4ex]
2.1. The minimal, disjoint part is $\U_{k,l}\oplus\V_{l'}<_{\mdeg}\U_{k,l'}\oplus \V_l$, such that $l<l'$:\\[0.2ex]
 The indecomposable $\V_{l'}$ can only be a direct summand of $M$ if $l'<s$ or $l'>t$.\\
 Thus, $l'>t$ and we obtain  $[\U_{t,l},M]<[\U_{t,l},L]$, a contradiction.\\[0.4ex]
2.2. The minimal, disjoint part is $\U_{k,l}\oplus\U_{l',k'}<_{\mdeg}\U_{k,l'}\oplus \U_{l,k'}$:\\[0.2ex]
 The indecomposable $\U_{l',k'}$ can only be a direct summand of $M$ if $l'>t$ or $k'<s$, or if $l'<s$ and $k'>t$. As has been shown in claim 1, every indecomposable $\U_{i,j}$ with $j<s$, or with $j>t$ and $i<s$ is either a direct summand of $M$, $L$ and $M'$ or a direct summand of none of them, thus, $l'>t$. But then we obtain $[\U_{t,l},M]<[\U_{t,l},L]$ , a contradiction.\\[0.4ex]
\textbf{3rd case:} The indecomposable $\U_{l,k}$ is a direct summand of $L$.\\[0.4ex]
Then $[\U_{1,t},M]<[\U_{1,t},L]$ if $s>1$ and $[\U_{t,n},M]<[\U_{t,n},L]$ if $t<n$. Of course, if $s=1$ and $t=n>2$, no representation $W$ as given in the assumption can exist at all, a contradiction.\\[0.4ex]
The assumption that $U_{k,l}$ is a direct summand of $L$, but not a direct summand of $M$ can be contradicted by a similar argumentation.\qedhere
\end{proof}
Claim 1 and claim 2 show that all arrows $l\rightarrow k$ with $b_{k,l}(M)=b_{k,l}(M')$ and $k,l\notin\{s,t\}$ coincide in $\olp(M)$, $\olp(L)$ and $\olp(M')$. 
The minimal, disjoint piece of the degeneration $D\oplus W<_{\mdeg}L$, therefore, has to be one of the following three.
\begin{itemize}
 \item $\U_{t,s} <_{\mdeg}\U_{s,t}$: Then $L\cong M'$, a contradiction to the assumption $L<_{\deg} M'$.
\item $\U_{t,s}\oplus \V_{k'} <_{\mdeg}\U_{t,k'}\oplus \V_s$ with $k'>t$: In this case $\U_{t,k'}\oplus \V_s\nless_{\deg}\U_{s,t}\oplus \V_{k'}$ and therefore  $L\nless_{\deg} M'$, a contradiction.
\item $\U_{t,s}\oplus \V_{k'} <_{\mdeg}\U_{k',s}\oplus \V_t$ with $k'<s$: In this case $\U_{k',s}\oplus \V_t\nless_{\deg}\U_{s,t}\oplus \V_{k'}$  and therefore  $L\nless_{\deg} M'$, a contradiction.
\end{itemize}
Since we obtain a contradiction in each case, the degeneration $M<_{\deg}M'$ is minimal.
\end{proof}
Note that in the setup of Theorem \ref{indecgeneral}, the condition $[X,D]-[X,D']=0$ is sufficient in most cases. The only exceptions are the minimal, disjoint degenerations 
 $D=\U_{s,t}<_{\mdeg}\V_s\oplus \V_t=D'$, such that $s<t$, and $D=\U_{r,t}\oplus\V_s<_{\mdeg}\U_{s,t}\oplus\V_r=D'$, such that $s<r$.\\[1ex]
 The concrete minimal degenerations are obtained easily from Proposition \ref{dimhom}. Furthermore, each minimal degeneration is of codimension $1$ (which is, as well, clear from the theory of spherical varieties, see \cite{Br1}).
\subsection{Parabolic orbit closures}
In case of the action of $P$, we describe all minimal, disjoint degenerations analogously to \cite[Theorem 4.6]{BoRe}.
\begin{theorem}\label{mindisjpar}
Let $D<_{\mdeg}D'$ be a minimal, disjoint degeneration in $\rep^{\inj}_{K}(\Q_p,I_2)$. Then it either appears in \cite[Theorem 4.6]{BoRe} or in one of the following chains.

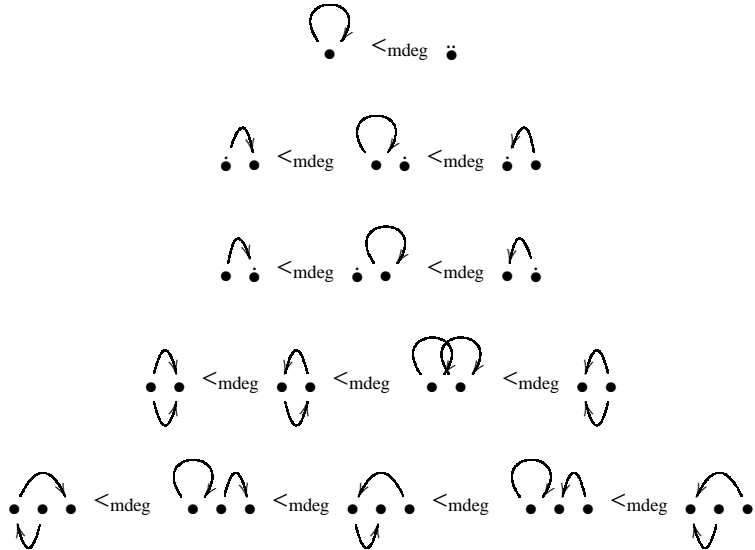
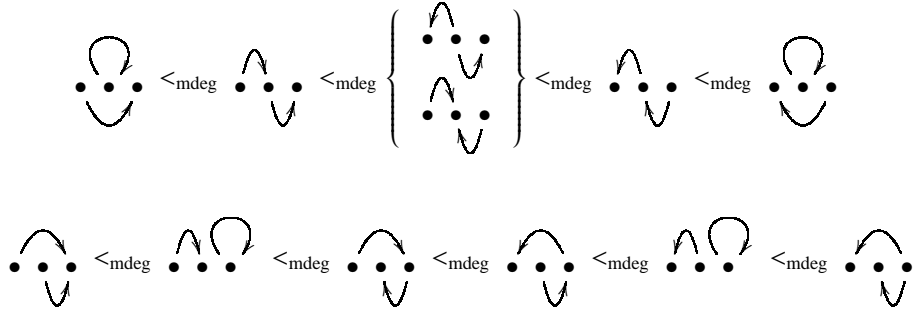
\[\xygraph{ !{<0cm,-0.13cm>;<0.63cm,-0.13cm>:<0cm,0.63cm>::} 
!{(0,0)}*+{\bullet}="1" 
"1":@(ul,ur) "1" }<_{\mdeg}
\xygraph{ !{<0cm,-0.13cm>;<0.63cm,-0.13cm>:<0cm,0.63cm>::} 
!{(0,0.03) }*+{\ddot{\bullet}}="1" 
 } \]

\[\xygraph{ !{<0cm,-0.13cm>;<0.63cm,-0.13cm>:<0cm,0.63cm>::} 
!{(-0.6,0.03) }*+{\dot{\bullet}}="1" 
!{(0,0) }*+{\bullet}="2" 
"1":@/^0.5cm/"2" 
 }<_{\mdeg}\xygraph{ !{<0cm,-0.13cm>;<0.63cm,-0.13cm>:<0cm,0.63cm>::} 
!{(-0.6,0)}*+{\bullet}="1" 
!{(0,0.03)}*+{\dot{\bullet}}="2" 
"1":@(ul,ur) "1" }<_{\mdeg}\xygraph{ !{<0cm,-0.13cm>;<0.63cm,-0.13cm>:<0cm,0.63cm>::} 
!{(-0.6,0.03) }*+{\dot{\bullet}}="1" 
!{(0,0) }*+{\bullet}="2" 
"2":@/_0.5cm/"1" 
 }\]

\[\xygraph{ !{<0cm,-0.13cm>;<0.63cm,-0.13cm>:<0cm,0.63cm>::} 
!{(-0.6,0) }*+{\bullet}="1" 
!{(0,0.03) }*+{\dot{\bullet}}="2" 
"1":@/^0.5cm/"2" 
 }<_{\mdeg}\xygraph{ !{<0cm,-0.13cm>;<0.63cm,-0.13cm>:<0cm,0.63cm>::} 
!{(-0.6,0.03)}*+{\dot{\bullet}}="1"
!{(0,0) }*+{\bullet}="2"
 "2":@(ul,ur) "2" }<_{\mdeg}\xygraph{ !{<0cm,-0.13cm>;<0.63cm,-0.13cm>:<0cm,0.63cm>::} 
!{(-0.6,0) }*+{\bullet}="1" 
!{(0,0.03) }*+{\dot{\bullet}}="2" 
"2":@/_0.5cm/"1" 
 }\]

\[ \xygraph{ !{<0cm,-0.13cm>;<0.63cm,-0.13cm>:<0cm,0.63cm>::} 
!{(-0.6,0) }*+{\bullet}="1" 
!{(0,0) }*+{\bullet}="2" 
"1":@/^0.5cm/"2" 
"1":@/_0.5cm/"2" 
 }<_{\mdeg} \xygraph{ !{<0cm,-0.13cm>;<0.63cm,-0.13cm>:<0cm,0.63cm>::} 
!{(-0.6,0) }*+{\bullet}="1" 
!{(0,0) }*+{\bullet}="2" 
"1":@/_0.5cm/"2" 
"2":@/_0.5cm/"1"}<_{\mdeg}
\xygraph{ !{<0cm,-0.13cm>;<0.63cm,-0.13cm>:<0cm,0.63cm>::} 
!{(-0.6,0)}*+{\bullet}="1"
!{(0,0) }*+{\bullet}="2"
 "1":@(ul,ur) "1"
 "2":@(ul,ur) "2" }<_{\mdeg}
\xygraph{ !{<0cm,-0.13cm>;<0.63cm,-0.13cm>:<0cm,0.63cm>::} 
!{(-0.6,0) }*+{\bullet}="1" 
!{(0,0) }*+{\bullet}="2" 
"2":@/^0.5cm/"1" 
"2":@/_0.5cm/"1" 
 }\]

\[\xygraph{!{<0cm,-0.13cm>;<0.63cm,-0.13cm>:<0cm,0.63cm>::} 
!{(-0.6,0) }*+{\bullet}="1" 
!{(0,0) }*+{\bullet}="2" 
!{(0.6,0) }*+{\bullet}="3" 
"1":@/^0.5cm/"3" 
"2":@/^0.5cm/"1"}<_{\mdeg} \xygraph{!{<0cm,-0.13cm>;<0.63cm,-0.13cm>:<0cm,0.63cm>::} 
!{(-0.6,0)}*+{\bullet}="1" 
!{(0,0)}*+{\bullet}="2" 
!{(0.6,0)}*+{\bullet}="3" 
 "1":@(ul,ur) "1"
"2":@/^0.5cm/"3"}
 <_{\mdeg}\xygraph{!{<0cm,-0.13cm>;<0.63cm,-0.13cm>:<0cm,0.63cm>::} 
!{(-0.6,0) }*+{\bullet}="1" 
!{(0,0) }*+{\bullet}="2" 
!{(0.6,0) }*+{\bullet}="3" 
"3":@/_0.5cm/"1" 
"1":@/_0.5cm/"2"}<_{\mdeg} \xygraph{!{<0cm,-0.13cm>;<0.63cm,-0.13cm>:<0cm,0.63cm>::} 
!{(-0.6,0) }*+{\bullet}="1"
!{(0,0) }*+{\bullet}="2"  
!{(0.6,0)}*+{\bullet}="3" 
 "1":@(ul,ur) "1"
"3":@/_0.5cm/"2"} 
 <_{\mdeg}\xygraph{!{<0cm,-0.13cm>;<0.63cm,-0.13cm>:<0cm,0.63cm>::} 
!{(-0.6,0) }*+{\bullet}="1" 
!{(0,0) }*+{\bullet}="2" 
!{(0.6,0) }*+{\bullet}="3" 
"2":@/^0.5cm/"1" 
"3":@/_0.5cm/"1"}\]

\[\xygraph{!{<0cm,-0.13cm>;<0.63cm,-0.13cm>:<0cm,0.63cm>::} 
!{(-0.6,0) }*+{\bullet}="1" 
!{(0,0) }*+{\bullet}="2" 
!{(0.6,0) }*+{\bullet}="3" 
 "2":@(ul,ur) "2"
"1":@/_0.5cm/"3" }<_{\mdeg}\xygraph{!{<0cm,-0.13cm>;<0.63cm,-0.13cm>:<0cm,0.63cm>::} 
!{(-0.6,0) }*+{\bullet}="1" 
!{(0,0) }*+{\bullet}="2" 
!{(0.6,0) }*+{\bullet}="3" 
"1":@/^0.5cm/"2" 
"2":@/_0.5cm/"3"}<_{\mdeg}
\left\lbrace \begin{array}{l}
\xygraph{!{<0cm,-0.13cm>;<0.63cm,-0.13cm>:<0cm,0.63cm>::} 
!{(-0.6,0) }*+{\bullet}="1" 
!{(0,0) }*+{\bullet}="2" 
!{(0.6,0) }*+{\bullet}="3" 
"2":@/_0.5cm/"1" 
"2":@/_0.5cm/"3"} \\
\xygraph{!{<0cm,-0.13cm>;<0.63cm,-0.13cm>:<0cm,0.63cm>::} 
!{(-0.6,0) }*+{\bullet}="1" 
!{(0,0) }*+{\bullet}="2" 
!{(0.6,0) }*+{\bullet}="3" 
"1":@/^0.5cm/"2" 
"3":@/^0.5cm/"2"}             \end{array}\right\rbrace <_{\mdeg}\xygraph{!{<0cm,-0.13cm>;<0.63cm,-0.13cm>:<0cm,0.63cm>::} 
!{(-0.6,0) }*+{\bullet}="1" 
!{(0,0) }*+{\bullet}="2" 
!{(0.6,0) }*+{\bullet}="3" 
"2":@/_0.5cm/"1" 
"3":@/^0.5cm/"2"} <_{\mdeg}\xygraph{!{<0cm,-0.13cm>;<0.63cm,-0.13cm>:<0cm,0.63cm>::} 
!{(-0.6,0) }*+{\bullet}="1" 
!{(0,0) }*+{\bullet}="2" 
!{(0.6,0) }*+{\bullet}="3" 
 "2":@(ul,ur) "2"
"3":@/^0.5cm/"1" }\]

\[\xygraph{!{<0cm,-0.13cm>;<0.63cm,-0.13cm>:<0cm,0.63cm>::} 
!{(-0.6,0) }*+{\bullet}="1" 
!{(0,0) }*+{\bullet}="2" 
!{(0.6,0) }*+{\bullet}="3" 
"1":@/^0.5cm/"3" 
"2":@/_0.5cm/"3"}<_{\mdeg} 
\xygraph{!{<0cm,-0.13cm>;<0.63cm,-0.13cm>:<0cm,0.63cm>::} 
!{(-0.6,0)}*+{\bullet}="1"
!{(0,0)}*+{\bullet}="2" 
!{(0.6,0)}*+{\bullet}="3" 
"1":@/^0.5cm/"2"
 "3":@(ul,ur) "3"}
 <_{\mdeg}
\xygraph{!{<0cm,-0.13cm>;<0.63cm,-0.13cm>:<0cm,0.63cm>::} 
!{(-0.6,0) }*+{\bullet}="1" 
!{(0,0) }*+{\bullet}="2" 
!{(0.6,0) }*+{\bullet}="3" 
"1":@/^0.5cm/"3" 
"3":@/^0.5cm/"2"}<_{\mdeg} 
\xygraph{!{<0cm,-0.13cm>;<0.63cm,-0.13cm>:<0cm,0.63cm>::} 
!{(-0.6,0) }*+{\bullet}="1" 
!{(0,0) }*+{\bullet}="2" 
!{(0.6,0) }*+{\bullet}="3" 
"2":@/_0.5cm/"3" 
"3":@/_0.5cm/"1"}<_{\mdeg} 
\xygraph{!{<0cm,-0.13cm>;<0.63cm,-0.13cm>:<0cm,0.63cm>::} 
!{(-0.6,0) }*+{\bullet}="1" 
!{(0,0) }*+{\bullet}="2" 
!{(0.6,0) }*+{\bullet}="3" 
"2":@/_0.5cm/"1"
 "3":@(ul,ur) "3"}
 <_{\mdeg}
\xygraph{!{<0cm,-0.13cm>;<0.63cm,-0.13cm>:<0cm,0.63cm>::} 
!{(-0.6,0) }*+{\bullet}="1" 
!{(0,0) }*+{\bullet}="2" 
!{(0.6,0) }*+{\bullet}="3" 
"3":@/^0.5cm/"2" 
"3":@/_0.5cm/"1"}\]
\end{theorem}
These minimal, disjoint degenerations yields concrete descriptions of the orbit closures in terms of enhanced oriented link patterns right away. 
\section{A finiteness criterion}\label{fincrit}
We consider the $P$-action on $\N_n^{(x)}$ and prove a criterion as to whether the action admits finitely many or infinitely many orbits. We call the parabolic subgroup $P$ maximal, if it is given by $2$ blocks $(b_1,b_2)$.
\begin{theorem}\label{classfinpar}
 There are only finitely many $P$-orbits in $\N_n^{(x)}$ if and only if $x\leq 2$, or $P$ is maximal and $x=3$.
\end{theorem}
\begin{proof}
If $x=2$, our considerations in Section \ref{x2} yield finiteness for every parabolic subgroup $P$.\\[1ex]
If $P$ is maximal of block sizes $(b_1,b_2)$ and $x=3$, then in order to calculate a system of representatives of the indecomposable representations, it suffices to calculate those of the quiver
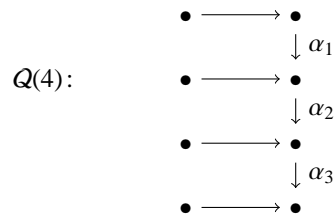
\begin{center}\begin{tikzpicture}
\matrix (m) [matrix of math nodes, row sep=1em,
column sep=3em, text height=1.5ex, text depth=0.2ex]
{ & \bullet  & \bullet \\ 
 \Q(4)\colon & \bullet  & \bullet \\
 & \bullet  & \bullet \\
 & \bullet  & \bullet \\};
\path[->]
(m-1-2) edge  (m-1-3)
(m-2-2) edge   (m-2-3)
(m-3-2) edge  (m-3-3)
(m-4-2) edge  (m-4-3)
(m-1-3) edge node[right=0.05cm] {$\alpha_{1}$} (m-2-3)
(m-2-3) edge node[right=0.05cm] {$\alpha_2$} (m-3-3)
(m-3-3) edge node[right=0.05cm] {$\alpha_{3}$} (m-4-3);\end{tikzpicture}\end{center}
with the relation  $\alpha_{3}\alpha_{2}\alpha_1=0$ which follows from the before mentioned covering theory of quivers. These can be calculated by using the knitting process and the  quiver in Figure \ref{fig1} is obtained. Note that we directly delete zero rows in the dimension vectors.\\[1ex]
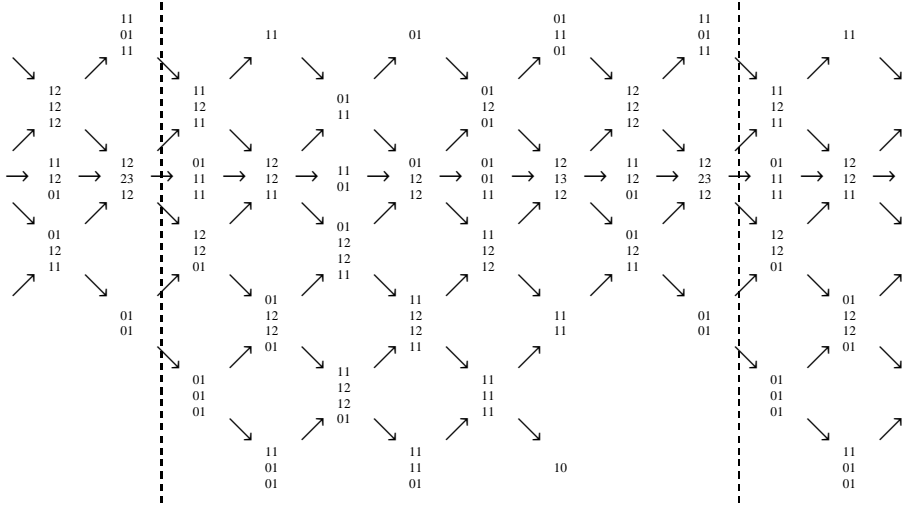
\begin{figure}[t]
\setlength{\unitlength}{0.8mm}
\begin{picture}(30,20)(-8,6)
\multiput(21,-65)(0,2){42}{\line(0,1){1}}
\multiput(117,-65)(0,2){42}{\line(0,1){1}}
 \put(-4,6){$\searrow$}
\put(-4,-6){$\nearrow$}\put(-4,-18){$\searrow$}\put(-5,-12){$\rightarrow$}
\put(-4,-30){$\nearrow$}
  \put(0,0){\begin{tiny}$\begin{array}{l}
1 2 \\ 
1 2 \\ 
1 2  \end{array}$\end{tiny}}\put(8,6){$\nearrow$}\put(8,-6){$\searrow$}
  \put(0,-12){\begin{tiny}$\begin{array}{l}
1 1 \\ 
1 2 \\ 
0 1  \end{array}$\end{tiny}}\put(7,-12){$\rightarrow$}
  \put(0,-24){\begin{tiny}$\begin{array}{l}
0 1 \\ 
1 2 \\ 
1 1  \end{array}$\end{tiny}}\put(8,-18){$\nearrow$}\put(8,-30){$\searrow$}
  \put(12,12){\begin{tiny}$\begin{array}{l}
1 1 \\ 
0 1 \\ 
1 1  \end{array}$\end{tiny}}
  \put(12,-12){\begin{tiny}$\begin{array}{l}
1 2 \\ 
2 3 \\ 
1 2  \end{array}$\end{tiny}}
  \put(12,-36){\begin{tiny}$\begin{array}{l}
0 1 \\ 
0 1  \end{array}$\end{tiny}}
  \put(20,6){$\searrow$}

 \put(20,-6){$\nearrow$}\put(20,-18){$\searrow$}\put(19,-12){$\rightarrow$}

  \put(20,-30){$\nearrow$}\put(20,-42){$\searrow$}
    \put(24,0){\begin{tiny}$\begin{array}{l}
1 1 \\ 
1 2 \\ 
1 1  \end{array}$\end{tiny}}\put(32,6){$\nearrow$}\put(32,-6){$\searrow$}
  \put(24,-12){\begin{tiny}$\begin{array}{l}
0 1 \\ 
1 1 \\ 
1 1  \end{array}$\end{tiny}}\put(31,-12){$\rightarrow$}
  \put(24,-24){\begin{tiny}$\begin{array}{l}
1 2 \\ 
1 2 \\ 
0 1  \end{array}$\end{tiny}}\put(32,-18){$\nearrow$}\put(32,-30){$\searrow$}
  \put(24,-48){\begin{tiny}$\begin{array}{l}
0 1 \\ 
0 1 \\ 
0 1  \end{array}$\end{tiny}}\put(32,-42){$\nearrow$}\put(32,-54){$\searrow$}
   \put(36,12){\begin{tiny}$\begin{array}{l}
1 1   \end{array}$\end{tiny}}\put(44,6){$\searrow$}

  \put(36,-12){\begin{tiny}$\begin{array}{l}
1 2 \\ 
1 2 \\ 
1 1  \end{array}$\end{tiny}}\put(44,-6){$\nearrow$}\put(44,-18){$\searrow$}\put(43,-12){$\rightarrow$}

  \put(36,-36){\begin{tiny}$\begin{array}{l}
0 1 \\ 
1 2 \\ 
1 2 \\
0 1 \end{array}$\end{tiny}}\put(44,-30){$\nearrow$}\put(44,-42){$\searrow$}
  \put(36,-60){\begin{tiny}$\begin{array}{l}
1 1 \\ 
0 1 \\ 
0 1  \end{array}$\end{tiny}}\put(44,-54){$\nearrow$}
  \put(48,0){\begin{tiny}$\begin{array}{l}
0 1 \\ 
1 1   \end{array}$\end{tiny}}\put(56,6){$\nearrow$}\put(56,-6){$\searrow$}
  \put(48,-12){\begin{tiny}$\begin{array}{l}
1 1 \\ 
0 1   \end{array}$\end{tiny}}\put(55,-12){$\rightarrow$}
  \put(48,-24){\begin{tiny}$\begin{array}{l}
0 1 \\ 
1 2 \\ 
1 2 \\
1 1 \end{array}$\end{tiny}}\put(56,-18){$\nearrow$}\put(56,-30){$\searrow$}

  \put(48,-48){\begin{tiny}$\begin{array}{l}
1 1 \\ 
1 2 \\ 
1 2 \\
0 1 \end{array}$\end{tiny}}\put(56,-42){$\nearrow$}\put(56,-54){$\searrow$}
  \put(60,12){\begin{tiny}$\begin{array}{l}
0 1  \end{array}$\end{tiny}}\put(68,6){$\searrow$}

  \put(60,-12){\begin{tiny}$\begin{array}{l}
0 1 \\ 
1 2 \\ 
1 2  \end{array}$\end{tiny}}\put(68,-6){$\nearrow$}\put(68,-18){$\searrow$}\put(67,-12){$\rightarrow$}

\put(60,-36){\begin{tiny}$\begin{array}{l}
1 1 \\ 
1 2 \\ 
1 2 \\
1 1 \end{array}$\end{tiny}}\put(68,-30){$\nearrow$}\put(68,-42){$\searrow$}
  \put(60,-60){\begin{tiny}$\begin{array}{l}
1 1 \\ 
1 1 \\ 
0 1  \end{array}$\end{tiny}}\put(68,-54){$\nearrow$}
  \put(72,0){\begin{tiny}$\begin{array}{l}
0 1 \\ 
1 2 \\ 
0 1  \end{array}$\end{tiny}}\put(80,6){$\nearrow$}\put(80,-6){$\searrow$}
  \put(72,-12){\begin{tiny}$\begin{array}{l}
0 1 \\ 
0 1 \\ 
1 1  \end{array}$\end{tiny}}\put(79,-12){$\rightarrow$}
  \put(72,-24){\begin{tiny}$\begin{array}{l}
1 1 \\ 
1 2 \\ 
1 2  \end{array}$\end{tiny}}\put(80,-18){$\nearrow$}\put(80,-30){$\searrow$}

  \put(72,-48){\begin{tiny}$\begin{array}{l}
1 1 \\ 
1 1 \\ 
1 1  \end{array}$\end{tiny}}\put(80,-42){$\nearrow$}\put(80,-54){$\searrow$}
   \put(84,12){\begin{tiny}$\begin{array}{l}
0 1 \\ 
1 1 \\ 
0 1  \end{array}$\end{tiny}}\put(92,6){$\searrow$}

  \put(84,-12){\begin{tiny}$\begin{array}{l}
1 2 \\ 
1 3 \\ 
1 2  \end{array}$\end{tiny}}\put(92,-6){$\nearrow$}\put(92,-18){$\searrow$}\put(91,-12){$\rightarrow$}

\put(84,-36){\begin{tiny}$\begin{array}{l}
1 1 \\ 
1 1  \end{array}$\end{tiny}}\put(92,-30){$\nearrow$}
  \put(84,-60){\begin{tiny}$\begin{array}{l}
1 0   \end{array}$\end{tiny}}

  \put(96,0){\begin{tiny}$\begin{array}{l}
1 2 \\ 
1 2 \\ 
1 2  \end{array}$\end{tiny}}\put(104,6){$\nearrow$}\put(104,-6){$\searrow$}
  \put(96,-12){\begin{tiny}$\begin{array}{l}
1 1 \\ 
1 2 \\ 
0 1  \end{array}$\end{tiny}}\put(103,-12){$\rightarrow$}
  \put(96,-24){\begin{tiny}$\begin{array}{l}
0 1 \\ 
1 2 \\ 
1 1  \end{array}$\end{tiny}}\put(104,-18){$\nearrow$}\put(104,-30){$\searrow$}

   \put(108,12){\begin{tiny}$\begin{array}{l}
1 1 \\ 
0 1 \\ 
1 1  \end{array}$\end{tiny}}\put(116,6){$\searrow$}

  \put(108,-12){\begin{tiny}$\begin{array}{l}
1 2 \\ 
2 3 \\
1 2  \end{array}$\end{tiny}}\put(116,-6){$\nearrow$}\put(116,-18){$\searrow$}\put(115,-12){$\rightarrow$}

\put(108,-36){\begin{tiny}$\begin{array}{l}
0 1 \\ 
0 1  \end{array}$\end{tiny}}\put(116,-30){$\nearrow$}\put(116,-42){$\searrow$}
  
  \put(120,0){\begin{tiny}$\begin{array}{l}
1 1 \\ 
1 2 \\ 
1 1  \end{array}$\end{tiny}}\put(128,6){$\nearrow$}\put(128,-6){$\searrow$}
  \put(120,-12){\begin{tiny}$\begin{array}{l}
0 1 \\ 
1 1 \\ 
1 1  \end{array}$\end{tiny}}\put(127,-12){$\rightarrow$}
  \put(120,-24){\begin{tiny}$\begin{array}{l}
1 2 \\ 
1 2 \\ 
0 1  \end{array}$\end{tiny}}\put(128,-18){$\nearrow$}\put(128,-30){$\searrow$}

  \put(120,-48){\begin{tiny}$\begin{array}{l}
0 1 \\ 
0 1 \\ 
0 1  \end{array}$\end{tiny}}\put(128,-42){$\nearrow$}\put(128,-54){$\searrow$}

  \put(132,12){\begin{tiny}$\begin{array}{l}
1 1  \end{array}$\end{tiny}}\put(140,6){$\searrow$}

  \put(132,-12){\begin{tiny}$\begin{array}{l}
1 2 \\ 
1 2 \\ 
1 1  \end{array}$\end{tiny}}\put(140,-6){$\nearrow$}\put(140,-18){$\searrow$}\put(139,-12){$\rightarrow$}

  \put(132,-36){\begin{tiny}$\begin{array}{l}
0 1 \\ 
1 2 \\ 
1 2 \\
0 1  \end{array}$\end{tiny}}\put(140,-30){$\nearrow$}\put(140,-42){$\searrow$}
  \put(132,-60){\begin{tiny}$\begin{array}{l}
1 1 \\ 
0 1 \\ 
0 1  \end{array}$\end{tiny}}\put(140,-54){$\nearrow$}
\end{picture}\rule{3mm}{0mm}\vspace{6.2cm}
\caption{The Auslander-Reiten quiver $\Gamma(\Q,I)$}\label{fig1}
\end{figure}
There are only finitely many isomorphism classes of indecomposable representations, thus the translation to the $P$-orbits in $\N_n^{(3)}$ which are given by the isomorphism classes in $\rep_K^{\inj}(\Q_2,I_3)(b_1,n)$ yields finitely many orbits. The orbits and their closures as well as the open orbit are described explicitly in \cite[Section 4.1]{B1}.\\[2ex]
Now let $P$ be a non-maximal parabolic subgroup and let $x\geq 3$. The action of $P$ on $\N_n^{(x)}$ admits infinitely many orbits, because 
\[(D_x(\lambda))_{i,j}=\left\lbrace\begin{array}{ll}
\lambda, & \textrm{if}~i=n~\textrm{and}~j=1; \\ 
1, & \textrm{if}~(1\leq i<n~\textrm{and}~j=1)~ \textrm{or}~(i=n~\textrm{and}~1\leq j<n);\\ 
0, & \textrm{otherwise}. 
                                   \end{array}\right.\]
yields a $1$-parameter family of pairwise non-$P$-conjugate matrices for $\lambda\in K^*$.\\[1ex]
If $P$ is a maximal parabolic subgroup of block sizes $(x,y)$, then the action of $P$ on $\N_n^{(4)}$  admits infinitely many orbits:
\begin{enumerate}
 \item If $x=s+2\geq 2$ and $y=t+2\geq 2$ for $s,t\leq 0$, then the matrices
\[(E^s(n,\lambda))_{i,j}:=\left\lbrace\begin{array}{ll} 
(E(\lambda))_{i-s,j-s}, & \textrm{if}~s+1\leq i,j\leq s+4; \\ 
0, & \textrm{otherwise}.
                                   \end{array}\right.\]
where \[E(\lambda)\coloneqq\left( \begin{array}{cccc}
0 & 0 & 0 & 0 \\ 
1 & 0 & 0 & 0 \\ 
1 & 1 & 0 & 0 \\ 
\lambda & 1 & 1 & 0
    \end{array}\right)\]
for $\lambda\in K^*$, induce a $1$-parameter family of pairwise non-$P$-conjugate matrices.
\item If (without loss of generality) $x=1$ and $y=n-1$, then for $\lambda\in K^*$, the matrices 
\[(F(n,\lambda))_{i,j}=\left\lbrace\begin{array}{ll}
(F(\lambda))_{i,j}, & \textrm{if}~1\leq i,j\leq 4; \\ 
0, & \textrm{otherwise}.
                                   \end{array}\right.\]
where \[F(\lambda)\coloneqq\left( \begin{array}{cccc}
1 & 1 & 0 & 0 \\ 
-1 & -1 & 0 & 0 \\ 
\lambda-1 &  \lambda & -1 & 1 \\ 
\lambda &\lambda-1 & -1 & 1
    \end{array}\right)\]
induce a $1$-parameter family of pairwise non-$P$-conjugate matrices.\qedhere
\end{enumerate}
 \end{proof}
Note that the algebra $\A(p,x)$ is either of finite or of wild representation type, but never of infinite tame representation type.
\section{Generic normal forms in the nilpotent cone}\label{gnfsect}
We discuss the $P$-action  on the nilpotent cone $\N:=\N_n^{(n)}$ now and introduce a generic normal form. We, thereby, generalize a generic normal form for the orbits of the Borel-action which is introduced in \cite{BoRe,Hal}. \\[1ex]
 Let $V$ be an $n$-dimensional $K$-vector space and denote the space of partial $p$-step flags of dimensions $\df$ by $\Fa_{\dfp}(V)$, that is, $\Fa_{\dfp}(V)$ contains flags
\[(0=F_0\subset F_1\subset \punkte \subset F_{p-2} \subset F_{p-1} \subset F_{p}=V),\]
such that $\dim_K F_i=d_i$. Let $\varphi$ be a nilpotent endomorphism of $V$ and consider pairs of a nilpotent endomorphism and a $p$-step flag up to base change in $V$, that is, up to the $\GL(V)$-action via $g.(F_*,\varphi) = (gF_*, g\varphi g^{-1})$.\\[1ex]
Let us fix a partial flag $F_*\in \Fa_{\dfp}(V)$ and a nilpotent endomorphism $\varphi$ of $V$.
\begin{lemma}\label{pargenlem}
 The following properties of the pair $(F_*, \varphi)$ are equivalent:
\begin{enumerate}
 \item $\dim_K \varphi^{n-d_k}(F_k)=d_k$ for every $k\in\{0,\punkte, p\}$,
\item there exists a basis $w_1,\punkte, w_n$ of $V$, such that  for  all    $ k\in\{1,\punkte, p\}$:
\begin{enumerate} 
 \item[($ \textrm{a}_k$)]  $F_k=\left\langle w_1,\punkte,w_{d_k}\right\rangle$
\end{enumerate}
 and for every  $ k\in\{2,\punkte, p\}$:
\begin{enumerate} 
\item[($ \textrm{b}_k$)]  $\varphi(w_x)=\left\lbrace \begin{array}{ll} 

w_{x+1}\mod\left\langle w_{d_1+2},\punkte,w_n\right\rangle, &  \textrm{if}~x<d_1;\\
w_{d_{k-1}+1} \mod \left\langle w_{d_{k}+1},\punkte, w_n\right\rangle, &  \textrm{if}~x=d_{k-1};\\
w_{x+1} \mod \left\langle w_{d_k+1},\punkte, w_n\right\rangle, &  \textrm{if}~d_{k-1}<x< d_k;\\
0, & \textrm{if}~ x=n.
                           \end{array}\right. $
\end{enumerate}
\end{enumerate}
\end{lemma}
\begin{proof}
 By \cite[Theorem 5.1]{BoRe}, we find a basis $u_1,\punkte, u_n$ of $V$ that is adapted to $F_*$ and such that
\[\varphi\left(u_x\right)=u_{x+1} \mod\left\langle u_{x+2},\punkte,u_n\right\rangle.\]
It is clear by the theorem of the Jordan normal form that we can modify this basis, such that
\[\varphi(u_{x}) =\left\lbrace 
\begin{array}{ll} 
u_{x+1} \mod \left\langle u_{d_k+1},\punkte, u_n\right\rangle, &  \textrm{if}~ d_{k-1}<x< d_k; \\ 
0 \mod \left\langle u_{d_k+1},\punkte, u_n\right\rangle, &  \textrm{if}~x=d_k.
                                \end{array}\right.\]
One can now verify the existence of the sought basis by adapting the given basis accordingly.
\end{proof}

We make use of Theorem \ref{pargenlem} in order to find a generic normal form in $\N$. Therefore, given $a,b\in\{0,\punkte,n\}$ and a matrix $N\in \N$, we define $N_{(a,b)}$ to be the submatrix formed by the last $a$ rows and the first $b$ columns of $N$. 
\begin{corollary}\label{hnf} The following conditions on a matrix $N\in\N$ are equivalent:
\begin{enumerate}
\item The first $d_k$ columns of $N^{n-d_k}$ are linearly independent for $k\in\{1,\punkte,p-1\}$ ,
\item the minor $\det ((N^{n-d_k})_{(d_k,d_k)})$ is non-zero for each $k\in\{1,\punkte,p-1\}$ ,
\item $N$ is $P$-conjugate to a unique matrix $H$, such that for all $k\in\{1,\punkte,p\}$:
\[H_{i,j}=\left\lbrace \begin{array}{ll}
0, & \textrm{if}~ i\leq j; \\ 
0, &  \textrm{if}~i=d_1+1~ \textrm{and}~ j<d_1;\\
0, &  \textrm{if}~ d_{k-1}+3\leq i\leq d_k ~  \textrm{and}~ d_{k-1}+1\leq j\leq d_k-2,~ \textrm{such~that}~ i>j+1;\\
0, &  \textrm{if}~ d_{k-1}+2\leq i\leq d_{k}~ \textrm{and}~j=d_{k-1}; \\
1, &  \textrm{if}~i=j+1. \\  
                      \end{array}\right.\]
\end{enumerate}
\end{corollary}
The normal form is sketched in Figure \ref{figgnf} where the block sizes are those of the parabolic subgroup $P$.
\begin{figure}[t]
\centering
$\left(\begin{array}{l;{1.5pt/1.7pt}l;{1.5pt/1.7pt}l;{1.5pt/1.7pt}l;{1.5pt/1.7pt}l}
\begin{array}{cccc}
0&\cdots&&0\\ 
1&&&\\ 
&\ddots&&\vdots\\ 
0&&1&0\end{array}  &~~~~~~~~~ 0 &~~~~~~~ 0~~~~~~~ &~~~~~~~~~ 0 &~~~~~~~~~ 0\\ \hdashline[1pt/2pt] 
\begin{array}{cccc}
0&\cdots&0&1\\ 
*&\cdots&*&0\\ 
\vdots&&\vdots&\vdots\\ 
*&\cdots&*&0\end{array} & \begin{array}{cccc}
0&\cdots&&0\\ 
1&&&\\ 
&\ddots&&\vdots\\ 
0&&1&0\end{array}  & ~~~~~~~0 &~~~~~~~~~ 0 &~~~~~~~~~ 0\\  \hdashline[1pt/2pt] 
\begin{array}{ccc}
*&~~\cdots~~&*\\ 
\vdots&&\vdots\\ 
*&~~\cdots~~&*\end{array}                           & \begin{array}{cccc}*&\cdots&*&1\\ 
*&\cdots&*&0\\ 
\vdots&&\vdots&\vdots \\  
 *&\cdots&*&0 \end{array}   & ~~~~~\ddots  &~~~~~~~~~ 0 &~~~~~~~~~ 0\\  \hdashline[1pt/2pt] 
\begin{array}{ccc}
*&~~\cdots~~&*\\ 
\vdots&&\vdots\\ 
*&~~\cdots~~&*\end{array}& \begin{array}{ccc}
*&~~\cdots~~&*\\ 
\vdots&&\vdots\\ 
*&~~\cdots~~&*\end{array}& ~~~~~\ddots  &\begin{array}{cccc}
0&\cdots&&0\\ 
1&&&\\ 
&\ddots&&\vdots\\ 
0&&1&0\end{array}  &~~~~~~~~~ 0\\ \hdashline[1pt/2pt] 
\begin{array}{ccc}
*&~~\cdots~~&*\\ 
\vdots&&\vdots\\ 
*&~~\cdots~~&*\end{array}& \begin{array}{ccc}
*&~~\cdots~~&*\\ 
\vdots&&\vdots\\ 
*&~~\cdots~~&*\end{array} & \begin{array}{ccc}
*&\cdots&*\\ 
\vdots&&\vdots\\ 
*&\cdots&*\end{array}    &\begin{array}{cccc}*&\cdots&*&1\\ 
*&\cdots&*&0\\ 
\vdots&&\vdots&\vdots \\  
*&\cdots&*&0 \end{array}  &\begin{array}{cccc}
0&\cdots&&0\\ 
1&&&\\ 
&\ddots&&\vdots\\ 
0&&1&0\end{array} \\  

          \end{array}\right)$
\caption{The generic parabolic normal form}\label{figgnf}
\end{figure}
As a direct consequence, the affine space
\[\Ha_B:=\{ H\in \N\mid  H_{i,j}=0~{\rm for}~i\leq j;~ H_{i+1,i}=1~{\rm for~all~}i\}\]
separates the $B$-orbits in $\N$ generically (that is, in the open subset $\N_B\coloneqq B.\Ha_B\subseteq \N$).
Moreover, the space
\[\Ha_U:=\{ H\in \N\mid  H_{i,j}=0~{\rm for}~i\leq j;~ H_{i+1,i}\neq 0~{\rm for~all~}i\}\]
separates the $U$-orbits in $\N$ generically  in the open subset $\N_U\coloneqq U.\Ha_U\subseteq \N$. 
\section{Generation of (semi-) invariant rings}\label{generation}
From now on, we consider the action of the Borel subgroup $B$ and the unipotent subgroup $U$ on the nilpotent cone $\N$. We define (semi-) invariants which generate the corresponding ring of (semi-) invariants (as we will see in Theorem \ref{genersemi}). Let us start by defining those Borel-semi-invariants introduced in \cite[Proposition 5.3]{BoRe}.\\[1ex]
Given $i\in\{1,\punkte,n\}$, we denote by $\omega_i\colon B\rightarrow\textbf{G}_m$ the character which is defined by $\omega_i\left(g\right)=g_{i,i}$; the $\omega_i$ form a basis for the group of characters of $B$.\\[1ex]
Let us fix integers $s,t\in\mathbf{N}$. For $i\in\{1,\punkte,s\}$ and $j\in\{1,\punkte,t\}$, we fix integers $a_i,a'_j\in\{1,\punkte,n\}$ with \mbox{$a_1+\punkte+a_s=a'_1+\punkte+a'_t=:r$} and polynomials $\Pa_{i,j}\left(x\right)\in K[x]$.\\[1ex]
Let $N\in \N$, then for all such $i$ and $j$ we consider the submatrices 
$\Pa_{i,j}\left(N\right)_{(a_i,a'_j)}\in K^{a_i\times a'_j}$  and form the $r\times r$-block matrix \[N^{\Pa}\coloneqq\left(\Pa_{i,j}\left(N\right)_{(a_i,a'_j)}\right)_{i,j},~{\rm where}~\Pa\coloneqq \left(\left(a_i\right)_i,\left(a'_j\right)_j,\left(\Pa_{i,j}\right)_{i,j}\right).\] 
\begin{proposition}\label{semiprop}
 For every datum $\Pa$ as above, the function 
\[f^{\Pa}\colon \N~\rightarrow~ K; ~~N~\mapsto~\det\left(N^{\Pa}\right)\]
 defines a $B$-semi-invariant regular function on $\N$ of weight \[\sum_{i=1}^s\left(\omega_{n-a_i+1}+\punkte+\omega_n\right)-\sum_{j=1}^t\left(\omega_1+\punkte+\omega_{a'_j}\right).\]
\end{proposition}
Note that the function $f^{\Pa}$ is also a $U$-invariant regular function on $\N$.
\begin{theorem}\label{genersemi}
The semi-invariant ring $K[\N]^B_*$ is generated by the semi-invariants of Proposition \ref{semiprop}.
\end{theorem}
\begin{proof}
First, we show $K[R^{\inj}_{\dfs}(\Q_n,I_x)]^{\GL_{\dfs}}_*\subseteq K[R_{\dfs}(\Q_n)]^{\GL_{\dfs}}_*$:\\[1ex]
The surjection $K[R_{\dfs}(\Q_n)]\rightarrow K[R_{\dfs}(\Q_n,I_x)]$ induces a surjection on the corresponding semi-invariant rings, since $\GL_{\dfs}$ is reductive. Furthermore, the codimension of $R_{\dfs}(\Q_n,I_x)\backslash R^{\inj}_{\dfs}(\Q_n,I_x)$ in $R_{\dfs}(\Q_n,I_x)$ is greater or equal than $2$, which yields the claim.\\[1ex] 
Following Lemma \ref{bijection}, we see that each $B$-semi-invariant $f$ on $\N$ is uniquely lifted to a $\GL_{\dfs}$-semi-invariant in $K[R^{\inj}_{\dfs}(\Q_n,I_x)]$. Theorem \ref{Schosemi} yields that $K[R_{\dfs}(\Q_n)]^{\GL_{\dfs}}_*$ is spanned by the determinantal semi-invariants $f_{\phi}$ defined in Subsection \ref{repsofalgebras}. Therefore, it suffices to prove that each determinantal semi-invariant, restricted to $R^{\inj}_{\dfs}(\Q_n,I_x)$, corresponds to one of the $B$-semi-invariants of Proposition \ref{semiprop}.\\[1ex]
Let us fix an arbitrary morphism in $\add \Q$, say 
\[\phi\colon \bigoplus_{j=1}^n O(j)^{x_j}\rightarrow \bigoplus_{i=1}^n O(i)^{y_i},\] such that
$h\coloneqq \sum_{j\in \Q_0}x_j \cdot j=\sum_{i\in \Q_0}y_i \cdot i.$  Then, by Section \ref{repsofalgebras}, we obtain a determinantal semi-invariant 
$f_{\phi}$.\\[1ex] 
The homomorphism spaces $P(j,i)$ between two objects $O(j)$ and $O(i)$ in $\add\Q$ are generated as $K$-vector spaces by
\[P(j,i)=\left\lbrace 
\begin{array}{ll}
0, & \textrm{if}~j>i; \\[1ex] 
\left\langle \rho_{j,i}\coloneqq\alpha_{i-1}\cdots \alpha_j\right\rangle, & \textrm{if}~j\leq i<n; \\[1ex] 
\left\langle \rho_{j,n}^{(k)}\coloneqq\alpha^k\alpha_{n-1}\cdots \alpha_j \mid k\in \mathbf{N}\cup\{0\}\right\rangle, & \textrm{if}~i=n.
                     \end{array}\right. \] 
The morphism $\phi$ is given by a $\sum_{i=1}^n y_i \times \sum_{j=1}^n x_j$-matrix $H$ with entries being morphisms between objects in $\add\Q$. 
We can view the matrix $H$ as an $n\times n$ block matrix $H=(H_{i,j})_{1\leq i,j\leq n}$ with $H_{i,j}\in K^{y_i\times x_j}$ for $i,j\in\{1,\punkte,  n\}$. Then
\[\left(H_{i,j}\right)_{k,l}= \left\lbrace
\begin{array}{ll}
0, & \textrm{if} ~ i<j; \\ 
\lambda^{k,l}_{i,j}\cdot \rho_{j,i}, & \textrm{for~some~} \lambda^{k,l}_{i,j}\in K~\textrm{if}~ j\leq i<n;\\ 
 \sum\limits_{h=0}^{\infty}\left(\lambda^{k,l}_{n,j}\right)_h\cdot \rho^{(h)}_{j,n}, &\textrm{for~some~}\left(\lambda^{k,l}_{n,j}\right)_h\in K~ \textrm{if}~ j\leq i=n.
                                        \end{array}\right.  \] 
Given an arbitrary matrix $N\in \N$, we reconsider the representation $M^N$ defined in Theorem \ref{bijection}.
 Since $\GL_{\dfs}$ acts transitively on $R_{\dfs}^{\inj}(\Q')$ with $\Q'$ being the linearly oriented quiver of Dynkin type $A_n$, we can examine the restricted semi-invariant on these representations $M^N$.\\[1ex]
The $B$-semi-invariant of $\N$ associated to $f_{\phi}$ via the translation of Lemma \ref{bijection} is given by
\[f^{\phi}\colon \N \rightarrow K;~ N \mapsto \det M^N(\phi).\]
The matrix 
\[M^N(\phi)=\left(M^N_{i,j}\right)_{1\leq i,j\leq n}\in K^{h\times h}\]
 is given as a block matrix
where each block 
\[M^N_{i,j}=  \left(\left(M^N_{i,j}\right)_{k,l}\right)_{\begin{subarray}{l}
1\leq k\leq y_i \\ 
1\leq l\leq x_j \end{subarray}}\in K^{iy_i\times jx_j}\]
is again a block matrix. The blocks of $M^N_{i,j}$ are given by

\[K^{i\times j} \ni \left(M^N_{i,j}\right)_{k,l}=\left\lbrace \begin{array}{ll}
0, & \textrm{if} ~ i<j; \\ 
\lambda^{k,l}_{i,j}\cdot E^{(i)}_{(i,j)}, & \textrm{if}~ j\leq i<n;\\ 
 \sum\limits_{h=0}^{\infty}\left(\lambda^{k,l}_{n,j}\right)_h\cdot \left(N^h\right)_{(n,j)}, & \textrm{if}~ j\leq i=n; 
                                        \end{array}\right.\]
such that $E^{(i)}\in K^{i\times i}$ is the identity matrix.
Note that if  $i,j\in\{1,\punkte,n\}$ and $i<n$, then $M_{i,j}^N=M_{i,j}^{N'}=:M_{i,j}$ for every pair of matrices $N,N'\in \N$.\\[1ex]
We can without loss of generality assume $y_1=\punkte=y_{n-1}=0$ which can, for example, be seen by induction on the index $i$ of $y_i$. This assumption is not necessary for the proof, but will shorten the remaining argumentation. Let us
define
 \[a\coloneqq(\underbrace{n,\punkte,n}_{=:a_{1},\punkte,a_{y_n}})~~~
\textrm{and}~~~a'\coloneqq(\underbrace{1,\punkte,1}_{=:a'_{1,1},\punkte,a'_{1,x_1}},\underbrace{2,\punkte,2}_{=:a'_{2,1},\punkte,a'_{2,x_2}},\punkte,\underbrace{n,\punkte,n}_{=:a'_{n,1},\punkte,a'_{n,x_n}}).\]
Furthermore, define for $j\in\{1,\punkte, n\}$ and for each pair of integers $k\in \{1,\punkte, y_n\}$ and $l\in\{1,\punkte, x_j\}$ the polynomial
 \[\Pa^{(k,l)}_{j}\coloneqq\sum_{h=0}^{\infty}\left(\lambda^{k,l}_{n,j}\right)_h\cdot X^{h}. \]
Let us denote $\Pa\coloneqq\left(a,a',\left(P_{j}^{(k,l)}\right)_{j,k,l}\right)$ and let $N\in \N$; it suffices to show $f^{\Pa}(N)=f^{\phi}(N)$:

\begin{align}
     f^{\phi}(N)~  ~=~\det M^N(\phi)&  ~=~ \det \left(M^N_{n,j}\right)_{1\leq j\leq n} ~=~ \det \left( \left(\left(M^N_{n,j}\right)_{k,l}\right)_{\begin{subarray}{l}
1\leq k\leq y_n \\ 
1\leq l\leq x_j \end{subarray}}\right)_{1\leq j\leq n}\nonumber\\
           & ~=~ \det \left(\left(\sum\limits_{h=0}^{\infty}\left(\lambda^{k,l}_{n,j}\right)_h\cdot \left(N^h\right)_{(n,j)}\right)_{\begin{subarray}{l}
1\leq k\leq y_n \\ 
1\leq l\leq x_j \end{subarray}}\right)_{1\leq j\leq n} \nonumber\\
 &~=~ \det \left(\left(\Pa^{\left(k,l\right)}_{j}(N)_{(n,j)}\right)_{\begin{subarray}{l}
1\leq k\leq y_n \\ 
1\leq l\leq x_j \end{subarray}}\right)_{1\leq j\leq n}~=~ \det N^{\Pa}~=~ f^{\Pa}(N) .\qedhere \nonumber
\end{align}
\end{proof}
\begin{corollary}\label{Ugeninv}
 The $U$-invariant ring $K[\N]^U$ is spanned by the induced $U$-invariants.
\end{corollary}
\section[About the algebraic U-quotient of the nilpotent cone]{About the algebraic $U$-quotient of the nilpotent cone}\label{Uquot}
We have seen that the $U$-invariant ring $K[\N]^U$ is spanned by the functions defined in Proposition \ref{semiprop}. At least for the cases $n=2,3$ the quotient criterion which we prove (in a more general setup) in the next subsection helps to provide the explicit structure of these rings.
\subsection{A quotient criterion}
Let $G$ be a reductive algebraic group and $U$ be a unipotent subgroup. Then $U$ acts on $G$ by right multiplication and Lemma \ref{Uinvfin} states that the $U$-invariant ring $K[G]^U$ is finitely generated as a $K$-algebra. Thus, an algebraic $U$-quotient of $G$, namely $G\quot U\coloneqq\Spec K[G]^U$, exists together with a dominant morphism $\pi_{G\quot U}\colon G\rightarrow G\quot U$ which is in general not surjective. Note that there is an element $\overline{e}\in G\quot U$, such that  $\pi_{G\quot U}(g)=g\overline{e}$ for all $g\in G$.\\[1ex]
The group $G$ acts on $G\quot U$ by left multiplication. Let $X$ be an affine $G$-variety and consider the diagonal operation of $G$ on the affine variety $G\quot U\times X$; we consider the natural $G$-equivariant morphism $\iota: X\rightarrow G\quot U\times X$.\\
 Let $\pi'\colon G\quot U\times X\rightarrow (G\quot U\times X)\quot G\coloneqq\Spec K[G\quot U\times X]^G$ be the associated algebraic $G$-quotient, then we obtain a morphism 
\[\rho\coloneqq \pi'\circ \iota: X\rightarrow  (G\quot U\times X)\quot G.\]
The morphism $\rho$ induces an isomorphism 
$\rho^*\colon (K[G]^U\otimes K[X])^G \rightarrow K[X]^U.$\\
Thus, $X\quot U\cong (G\quot U\times X)\quot G$ and
\[K[X]^U \cong (K[G\quot U\times X])^G\cong (K[G\quot U]\otimes K[X])^G\cong (K[G]^U\otimes K[X])^G.\]
Let $Y$ be an affine $G$-variety and let $\mu': G\quot U\times X\rightarrow Y$ be a $G$-invariant morphism, together with a dominant $U$-invariant morphism of affine varieties 
\[\mu\colon X \rightarrow Y;~x  \mapsto (f_{1}(x),\punkte, f_s(x)),\] such that $\mu'\circ \iota = \mu$.\\[1ex]
In this setting, we obtain the following criterion for $\mu$ to be an algebraic $U$-quotient.
\begin{lemma}\label{critU}Assume that
\begin{enumerate}
 \item[(1.)] $Y$ is normal,
 \item[(2.)] $\mu$ separates the $U$-orbits generically, that is, there is an open subset $Y_U\subseteq Y$, such that $\mu(x)\neq \mu(x')$ for all $x,x'\in X_U\coloneqq\mu^{-1}(Y_U)$, and
\item[(3.)] $\codim_Y(\overline{Y\backslash Y_U})\geq 2$ or $\mu$ is surjective.
\end{enumerate}
Then $\mu$ is an algebraic $U$-quotient of $X$, that is, $Y\cong X\quot U$.
\end{lemma}
\begin{proof}
Let $g_{1},\punkte,g_s\in K[G\quot U\times X]^G$, such that $\rho^*(g_{i})=f_{i}$ for all $i$.\\[1ex]
Clearly,
\[2\leq \codim_Y(\overline{Y\backslash \mu(X)})\leq\codim_{Y}(\overline{Y\backslash \mu'(G\quot U\times X)}).\]
The morphism $\mu'$ separates the $G$-orbits in  $ G\quot U\times X$ generically (that is, in $Y_U$):\\
If $x,x'\in G.(\{\overline{e}\}\times X_U)$, then $\mu'(x)\neq \mu'(x')$. The morphism $\mu'$ restricts to a surjection $G.(\{\overline{e}\}\times X_U)\rightarrow Y_U$, furthermore, the algebraic quotient $\pi'$ is surjective and there exists a morphism $\tilde{\mu'}: (G\quot U\times X)\quot G\rightarrow Y$, such that $\tilde{\mu'}\circ \pi'=\mu'$. Then $Y_U\subseteq \im(\tilde{\mu'})$ and, since each fibre of $\pi'$  contains exactly one closed $G$-orbit, we have shown that generically each fibre of $\mu'$ contains a unique   closed orbit. \\[1ex]
Thus, Theorem \ref{criterion} yields that $\pi\colon G\quot U\times X\rightarrow Y$ is an algebraic $G$-quotient. Since $f_{i}$ and $g_{i}$ correspond to each other via the isomorphism $ \rho^*\colon (K[G]^U\otimes K[X])^G \rightarrow K[X]^U,$ the morphism $\mu\colon X\rightarrow Y$ is an algebraic $U$-quotient of $X$.
\end{proof}
We are now able to give explicit descriptions of algebraic $U$-quotients of the nilpotent cone in case $n$ equals $2$ or $3$.
\begin{example}\label{Utwo}
 We consider $\N=\N_2$. In this case, the $U$-normal form of Section \ref{gnfsect} is given by matrices  
 \[H_{x}\coloneqq\left( \begin{array}{ll}
0 &0  \\ 
x & 0
\end{array}\right)\]
 where $x\in K^*$. Then by Proposition \ref{semiprop}, we define the $U$-invariant $f_{2,1}$ by $f_{2,1}(N)=N_{2,1}$ for $N=(N_{i,j})_{i,j}\in\N$.\\[1ex]
 The morphism 
\[\mu\colon \N\rightarrow \textbf{A}^1= \Spec K[f_{2,1}];~N~\mapsto f_{2,1}(N)\]
is an algebraic $U$-quotient of $\N$:\\[1ex]
Clearly, the variety $\textbf{A}^1$ is normal and $\mu$ separates the $U$-orbits in the open subset $\N_U\subseteq \N$.
 Since $\mu$ is surjective, Theorem \ref{critU} yields the claim. We have, therefore, proven 
\[K[\N]^U=K[f_{2,1}].\]
\end{example}
The case $n=3$ is slightly more complex, but can still be handled by making use of Theorem \ref{critU}.
\begin{example}\label{Uthree}
 In  case $\N=\N_3$,  the $U$-normal forms are given by matrices \[H=\left( \begin{array}{lll}
0&0 &0  \\ 
x_1&0 & 0\\
x& x_2& 0
\end{array}\right),~x_1,x_2\in K^*.\] Following Proposition \ref{semiprop}, we define certain $U$-invariants; consider $N=(N_{i,j})_{i,j}\in \N$, then $f_{3,1}(N)= N_{3,1}$, ${\det}_1(N)= N_{2,1}N_{3,2}-N_{2,2}N_{3,1}$ and ${\det}_2(N) =N_{1,1}N_{3,1}+N_{2,1}N_{3,2}+N_{3,1}N_{3,3}$. Note that the equality $\Det_1(N)=\Det_2(N)$ holds true for all $N\in\N$ due to the nilpotency conditions.\\[1ex]
Furthermore, we define a $U$-invariant $f_1$ given by the datum $\Pa=((2),(1,1), (x,x^2))$, thus, 
 $f_{1}(N)= N_{2,1}\cdot {\det}_1+N_{3,1}\cdot(N_{2,1}N_{3,3}-N_{3,1}N_{2,3})$.\\[1ex]
And the $U$-invariant $f_2$ given by the datum $\Pa=((1,1),(2),(x^2,x)$, thus,
 $f_{2}(N) = N_{3,2}\cdot {\det}_1+N_{3,1}\cdot(N_{1,1}N_{3,2}-N_{1,2}N_{3,1})$. Then $f_1\cdot f_2=\Det_1^3$ holds true in $\N$.\\[1ex]
\textbf{Claim:}
 The morphism 
 \begin{align}
 \mu\colon& ~\N\rightarrow \textbf{A}^1\times \Spec \dfrac{K[X_1,X_2,Z]}{\left( X_1X_2=Z^3\right)}=: Y\nonumber\\
&~N~\mapsto (f_{3,1}(N), f_1(N),f_2(N), {\det}_1(N)) \nonumber
\end{align}
is an algebraic $U$-quotient of $\N$.\\[2ex]
The affine variety $Y$ is normal as the product of $\textbf{A}^1$ and a normal toric affine variety induced by the strongly convex rational polyhedral cone \[\sigma\coloneqq \Cone\left(
\left(\begin{array}{l}
1 \\ 
1     \end{array}\right),\left(\begin{array}{l}
1 \\ 
2     \end{array}\right),\left(\begin{array}{l}
2 \\ 
1     \end{array}\right)
 \right). \] 
The morphism $\mu$ separates the $U$-orbits in the open subset $\N_U\subseteq \N$ as can be proved by a direct calculation.\\[1ex]
Furthermore, $\codim_{Y}(\overline{Y\backslash \mu(\N)})\geq 2$, since $\textbf{A}^1\times  X'\subset \mu(\N)$ and $(s,t,u,v)\in \mu(\N)$ whenever either $s$, $t$ or $u$ equals zero and $v^3=ut$.
Theorem \ref{critU} yields the claim.\\[1ex]
We have proved 
\[K[\N]^U=\dfrac{K[f_{3,1}, f_1, f_2, {\det}_1]}{\left( f_1\cdot f_2= {\det}_1^3\right)}.\]
\end{example}
\subsection{Toric invariants}\label{toricsect}
As the case $n=3$ suggests, there is a toric variety closely related to $\N\quot U$.\\[1ex]
 By considering a special type of $U$-invariants, so-called toric invariants, we define a toric variety $X$ together with a dominant morphism
$ \N\quot U\rightarrow X$, such that the generic fibres are affine spaces of the same dimension.\\[1ex]
Given a matrix $H=(x_{i,j})_{i,j}\in \Ha_U$, we denote $x_i\coloneqq x_{i+1,i}$ and define its \textit{toric part} $H_{\tor}\in K^{n\times n}$ by 
\[(H_{\tor})_{i,j}\coloneqq\left\lbrace
\begin{array}{ll}
x_{i}, & \textrm{if}~i=j+1; \\ 
0, & \textrm{otherwise}.
\end{array}\right. \] 
Let $f\neq 0$ be a invariant, given by the data \[\Pa=((a_i)_{1\leq i\leq s},(a'_j)_{1\leq j\leq t},(\Pa_{i,j})_{\begin{subarray}{l}
1\leq i\leq s \\ 
1\leq j\leq t \end{subarray}}).\]
We call $f$ \textit{toric} if $f(H)=f(H_{\tor})$ for every matrix $H\in \Ha_U$ and \textit{sum-free} if its block sizes  $a_{1},\punkte, a_{s}$ and $a'_{1}, \punkte , a'_{t}$ do not share any partial sums, that is, $\sum_{i\in I} a_i \neq \sum_{i'\in I'}a'_{i'}$ for all $I\subsetneq \{1,\punkte,s\}$ and $I'\subsetneq \{1,\punkte,t\}$.
\begin{lemma}\label{secred}
The toric invariants are generated by the sum-free toric invariants.
\end{lemma}
The proof is provided by double induction on the integers $s$ and $t$ and can be found in \cite[Lemma 6.2.3]{Bo1}.\\[1ex]
We  denote the subring of $K[\N]^U$ which is generated by all toric invariants by $K[\N]^U_{\tor}$. 
Corresponding to $K[\N]^U_{\tor}$, there is a variety $X\coloneqq \Spec K[\N]^U_{\tor}$ which is a toric variety.
Given a sum-free toric invariant, there are integers $h_1,\punkte,h_{n-1}$, such that \[f(H)=x_1^{h_1}\cdot \punkte\cdot x_{n-1}^{h_{n-1}}.\]
Denote by $S$ the set of tuples $(h_1,\punkte,h_{n-1})\in \textbf{N}^{n-1}$ that arise in this way from a minimal set of generating toric invariants and denote $\sigma\coloneqq \Cone(S)$.\\[1ex]
Let $N$ be the lattice $\textbf{Z}^{n-1}$, then $\sigma$ is generated by the finite set $S\subset \textbf{Z}^{n-1}$ and fulfills $\sigma \cap (-\sigma)=\{0\}$, therefore, $\sigma$ as well as $\sigma^\vee$ are strongly convex rational polyhedral cones of maximal dimension. The variety $X= \Spec K[\N]^U_{\tor}\cong\Spec K[S_{\sigma^{\vee}}]$, thus, is a normal toric variety by Lemma \ref{toricco}.\\[1ex]
Let $T\subset \GL_n$ be the torus of diagonal matrices. There is a natural action $\tau$ of $T$ on the $U$-invariant ring of $\N$ as follows: 
\[\tau\colon ~ T\times K[\N]^U \rightarrow K[\N]^U;~ (t,f)~  \mapsto \left( \begin{array}{ll}
f\colon & \N \rightarrow K \\ 
 & N\mapsto f(tNt^{-1})
                 \end{array}\right).\]
Another operation is given, since the variety $X=\Spec K[\N]^U_{\tor}$ is a toric variety:
\begin{align}
\tau'\colon &~ (K^*)^{n-1}\times K[\N]^U_{\tor} \rightarrow K[\N]^U_{\tor}.\nonumber
\end{align} 
Let $f$ be a toric invariant, such that $f(H)=x_1^{h_1}\punkte x_{n-1}^{h_{n-1}}$, and let $c\coloneqq(c_1,\punkte,c_{n-1})\in(K^*)^{n-1}$. Then
$\tau'(c,f)(H)= f(H)\cdot c_1^{h_1}\punkte c_{n-1}^{h_{n-1}}.$\\[1ex]
The  operation $\tau$ is induced by the operation $\tau'$ via the morphism
\[\rho\colon ~  T \rightarrow (K^*)^{n-1};~ (t_1,\punkte,t_n)\mapsto (t_{2}/t_1,\punkte,t_{n}/t_{n-1}).\]
\begin{lemma}
 Let  $f$ be a sum-free toric invariant of block sizes $\underline{a}\coloneqq(a_1,\punkte,a_s)$ and $\underline{a}'\coloneqq(a'_1,\punkte,a'_t)$ and let $f(H)= x_1^{h_1} \punkte x_{n-1}^{h_{n-1}} $. Then $h_{n-1}=s$ and for $l\in \{1,\punkte,n-2\}$:
 \[h_l= t+\sum_{k=2}^l \sharp\{j\in\{1,\punkte,t\}\mid a'_j\geq k\} -\sum_{k=1}^{l-1} \sharp\{i\in\{1,\punkte,s\}\mid a_i\geq n-k\}\]
\end{lemma}
Let $i\in\{1,\punkte,n-1\}$, then we define the $U$-invariant $\det_{i}(N):=\Det(N^{n-i}_{(i,i)})$ and the $U$-invariant $f_i$ to be the unique toric invariant of block sizes $(i),(1,\punkte,1)$. 
 Furthermore, for integers $i,j\in\{1,\punkte,n\}$, such that $j< i-1$, we define the datum \[\Pa=\left((j-1,n-i+1),(j,n-i),\left(\begin{array}{cc}
x^{n-j+1} & 0 \\ 
x & x^i
                         \end{array} \right)\right)\] 
and denote $f_{i,j}\coloneqq f^{\Pa}$. These invariants separate the $U$-orbits generically in $\N_U\subseteq \N$.\\[1ex]
Let $\pi: \N \rightarrow \N\quot U$ be an algebraic $U$-quotient of $\N$ which exists, since $K[\N]^U$ is finitely generated. The variety $\N\quot U$ is normal, since the nilpotent cone is normal (see \cite[III.3.3]{Kr}).\\[1ex]
The space of $U$-normal forms is given by $\Ha_U\cong \textbf{A}^D\times (K^*)^{n-1}$ and the map $\pi$ restricts to a morphism $i:\Ha_U\rightarrow \N\quot U$. We consider the toric variety $X$ described above by its cone $\sigma$ which is induced by the sum-free toric invariants and let $X'\cong (K^*)^{n-1}$ be the dense orbit in $X$.\\[1ex]
The morphism $i: \Ha_U\rightarrow i(\Ha_U)$ is injective, since   the fibres are separated generically by certain $U$-invariants. Therefore, we can construct an explicit morphism $i': i(\Ha_U)\rightarrow \Ha_U$, such that $i\circ i'=\id_{ i(\Ha_U)}$ and $i'\circ i=\id_{ \Ha_U}$. The morphism $i$ is, thus, birational  and $\textbf{A}^D\times (K^*)^{n-1}\cong i(\Ha_U)\subseteq \N\quot U$. 
\begin{lemma}
 The natural embedding $K[\N]^U_{\tor}\rightarrow K[\N]^U$ induces a dominant, $T$-equivariant morphism $p: \N\quot U\rightarrow X,$  such that $p^{-1}(x)\cong \textbf{A}^D$ for each point $x'\in X'$. 
\end{lemma}
\begin{proof}
The morphism $p$ is clearly dominant and $T$-equivariant due to our considerations above. \\[1ex]
Let $x'\in X'$, then $p^{-1}(x)\subseteq i(\Ha_U)$, since every determinant ${\det}_i$ for $i\in\{1,\punkte,n-1\}$ is a toric invariant. If $x'\in X'$, none of these determinants vanishes on $x'$ and Section \ref{gnfsect}, therefore, yields $p^{-1}(x')\subseteq i(\Ha_U)$. Since the orbits in $\N_U$ are separated by certain $U$-invariants  and since $\Ha_U\cong \textbf{A}^D\times X'$, the claim  $p^{-1}(x)\cong \textbf{A}^D$ follows.
\end{proof}
 There is a morphism $q: \N\quot U\rightarrow \textbf{A}^D$ as well, such that the composition
\[\Ha_U\xrightarrow{i} \N\quot U\xrightarrow{q} \textbf{A}^D\]
yields $q\circ i(H)= (x_{i,j})_{1< j+1\leq i-1< n}\in \textbf{A}^D$.

\begin{lemma}
 The morphism
\[(q,p): \N\quot U\rightarrow \textbf{A}^D\times X\]
is dominant and birational.
\end{lemma}
\begin{proof}
 The morphism $(p,q)$ is dominant, since  $\textbf{A}^D\times X'\subseteq \im(p,q)\subseteq \textbf{A}^D\times X$.\\[1ex]
The morphism $(p,q)$ is birational, since $(p,q)$ is dominant and generically injective: the fibre $(p,q)^{-1}(y)$ contains exactly one element for every $y\in \textbf{A}^D\times X'$, since the $U$-orbits can be separated in $\textbf{A}\times X'$. More straight forward, $(p,q)$ restricts to an isomorphism $i(\Ha_U)\cong \textbf{A}^D\times X'$.
\end{proof}
Note that the morphism $(p,q)$ is not surjective for $n\geq 4$. Even in the case $n=4$, we can show $K[\N]^U\ncong K[\textbf{A}^3]\otimes K[\N]^U_{\tor}$ and $\N\quot U\ncong \textbf{A}^3\times X$.\\[1ex]
We define a $U$-invariant $g$ by the data 
\[\Pa=\left\lbrace 
\begin{array}{ll}
((2),(2), (x)), &~\textrm{if}~n=4; \\ 
((n-2),(2,n-4), (x,x^4)) &~\textrm{otherwise}.
\end{array}
\right. \]   
Then 
$g(H)=(x_{3,1}\cdot x_{4,2}-x_2\cdot x_{4,1})\cdot \Det_{n-4}(H)$
and the relation 
\[g\cdot \underbrace{\Det_{n-3}\cdot\Det_1\cdot f_{n-3}\cdot f_{n-1}}_{\coloneqq F} =\underbrace{f_{3,1}\cdot f_{4,2}\cdot f_{n-3} \cdot f_{n-1} - f_{4,1}\cdot f_{n-2}^2 \cdot \Det_{n-3}\cdot \Det_1}_{\coloneqq F'}\]
holds true in $K[\N]^U$. The set 
$M\coloneqq \{\underline{x}\in\textbf{A}^D\times X\mid F(\underline{x})= 0; F'(\underline{x})\neq 0\}$ is non-empty and 
 the inclusion $M \subseteq (\textbf{A}^D\times X)\backslash \im(p,q)$ directly yields that the morphism $(p,q)$ is not surjective.

\bibliography{Literatur}

\begin{thebibliography}{10}

\bibitem{ASS}
Ibrahim Assem, Daniel Simson, and Andrzej Skowro{\'n}ski.
\newblock {\em Elements of the representation theory of associative algebras.
  {V}ol. 1}, volume~65 of {\em London Mathematical Society Student Texts}.
\newblock Cambridge University Press, Cambridge, 2006.
\newblock Techniques of representation theory.

\bibitem{Bo1}
Klaus Bongartz.
\newblock Minimal singularities for representations of {D}ynkin quivers.
\newblock {\em Comment. Math. Helv.}, 69(4):575--611, 1994.

\bibitem{BoGa}
Klaus Bongartz and Peter Gabriel.
\newblock Covering spaces in representation-theory.
\newblock {\em Invent. Math.}, 65(3):331--378, 1981/82.

\bibitem{B1}
Magdalena Boos.
\newblock {\em Conjugation on varieties of nilpotent matrices}.
\newblock PhD thesis, Bergi\-sche Universit\"at, Wuppertal, 2012.

\bibitem{BoRe}
Magdalena Boos and Markus Reineke.
\newblock B-orbits of 2-nilpotent matrices.
\newblock In {\em Highlights in Lie Algebraic Methods (Progress in
  Mathematics)}, pages 147--166. Birkh\"auser, Boston, 2011.

\bibitem{Br1}
Michel Brion.
\newblock Spherical varieties.
\newblock In {\em Proceedings of the {I}nternational {C}ongress of
  {M}athematicians, {V}ol.\ 1, 2 ({Z}\"urich, 1994)}, pages 753--760, Basel,
  1995. Birkh\"auser.

\bibitem{Fr}
Lucas Fresse.
\newblock On the singular locus of certain subvarieties of springer fibres.
\newblock To appear in: Mathematical Research Letters, 2012.

\bibitem{Fu}
William Fulton.
\newblock {\em Introduction to toric varieties}, volume 131 of {\em Annals of
  Mathematics Studies}.
\newblock Princeton University Press, Princeton, NJ, 1993.
\newblock The William H. Roever Lectures in Geometry.

\bibitem{Ga3}
Peter Gabriel.
\newblock The universal cover of a representation-finite algebra.
\newblock In {\em Representations of algebras ({P}uebla, 1980)}, volume 903 of
  {\em Lecture Notes in Math.}, pages 68--105. Springer, Berlin, 1981.

\bibitem{Hal}
Bettina Halbach.
\newblock ${B}$-{O}rbiten nilpotenter {M}atrizen.
\newblock Bachelorarbeit, Bergische Universit\"at Wuppertal, 2009.

\bibitem{Hi}
David Hilbert.
\newblock Ueber die {T}heorie der algebraischen {F}ormen.
\newblock {\em Math. Ann.}, 36(4):473--534, 1890.

\bibitem{HiRoe}
Lutz Hille and Gerhard R{\"o}hrle.
\newblock A classification of parabolic subgroups of classical groups with a
  finite number of orbits on the unipotent radical.
\newblock {\em Transform. Groups}, 4(1):35--52, 1999.

\bibitem{Kr}
Hanspeter Kraft.
\newblock {\em Geometrische {M}ethoden in der {I}nvariantentheorie}.
\newblock Aspects of Mathematics, D1. Friedr. Vieweg \& Sohn, Braunschweig,
  1984.

\bibitem{Me1}
Anna Melnikov.
\newblock {$B$}-orbits in solutions to the equation {$X^2=0$} in triangular
  matrices.
\newblock {\em J. Algebra}, 223(1):101--108, 2000.

\bibitem{Me2}
Anna Melnikov.
\newblock Description of {B}-orbit closures of order 2 in upper-triangular
  matrices.
\newblock {\em Transform. Groups}, 11(2):217--247, 2006.

\bibitem{Mu}
Shigeru Mukai.
\newblock {\em An introduction to invariants and moduli}, volume~81 of {\em
  Cambridge Studies in Advanced Mathematics}.
\newblock Cambridge University Press, Cambridge, 2003.
\newblock Translated from the 1998 and 2000 Japanese editions by W. M. Oxbury.

\bibitem{Na1}
Masayoshi Nagata.
\newblock On the fourteenth problem of {H}ilbert.
\newblock In {\em Proc. {I}nternat. {C}ongress {M}ath. 1958}, pages 459--462.
  Cambridge Univ. Press, New York, 1960.

\bibitem{Pan1}
Dmitri~I. Panyushev.
\newblock Complexity and nilpotent orbits.
\newblock {\em Manuscripta Math.}, 83(3-4):223--237, 1994.

\bibitem{SvB}
Aidan Schofield and Michel van~den Bergh.
\newblock Semi-invariants of quivers for arbitrary dimension vectors.
\newblock {\em Indag. Math. (N.S.)}, 12(1):125--138, 2001.

\end{thebibliography}
\end{document}